\documentclass{article}

\usepackage{arxiv}
\usepackage{comment}
\usepackage{amsmath}
\usepackage{hyperref}
\usepackage{amsthm}
\usepackage{tikz-cd}

\usepackage[utf8]{inputenc} 
\usepackage[T1]{fontenc}    

\usepackage{url}            
\usepackage{booktabs}       
\usepackage{amsfonts}       
\usepackage{nicefrac}       
\usepackage{microtype}      
\usepackage{lipsum}		
\usepackage[inkscapearea=page]{svg}
\usepackage{graphicx}
\usepackage{hyperref}       
\usepackage{doi}
\usepackage{xcolor}
\usepackage[font=scriptsize,labelfont=bf]{caption}
\usepackage{textgreek}
\usepackage{subcaption}

\newtheorem{theorem}{\bf Theorem}[section]

\pdfoutput=1

\title{Riemannian Geometry and Molecular Similarity II: K\"ahler Quantization}

\author{Rachael Pirie \\
    School of Natural and Environmental Sciences\\
    Newcastle University\\
	\texttt{r.pirie2@ncl.ac.uk}\\
	\And
	{Stuart J. Hall} \\
	School of Mathematics, Statistics, and Physics\\
	Newcastle  University\\
	\texttt{stuart.hall@ncl.ac.uk} \\
	\AND
	{Daniel J. Cole}\\
School of Natural and Environmental Sciences\\
Newcastle University\\
	\texttt{daniel.cole@ncl.ac.uk}
	\And
	 {Thomas Murphy} \\
	Department of Mathematics\\
	California State University, Fullerton \\
	\texttt{tmurphy@fullerton.edu} \\
}

\renewcommand{\shorttitle}{Geometry and Molecular Surfaces}

\hypersetup{
pdftitle={A template for the arxiv style},
pdfsubject={q-bio.NC, q-bio.QM},
pdfauthor={David S.~Hippocampus, Elias D.~Striatum},
pdfkeywords={First keyword, Second keyword, More},
}

\begin{document}
\maketitle
\noindent
\begin{abstract}
Shape-similarity between molecules is a tool used by chemists for virtual screening, with the goal of reducing the cost and duration of drug discovery campaigns. This paper reports an entirely novel shape descriptor as an alternative to the previously described RGMolSA descriptors \cite{cole2022riemannian}, derived from the theory of Riemannian geometry and K\"ahler quantization (KQMolSA). The treatment of a molecule as a series of intersecting spheres allows us to obtain the explicit \textit{Riemannian metric} which captures the geometry of the surface, which can in turn be used to calculate a Hermitian matrix $\mathbb{M}$ as a directly comparable surface representation. The potential utility of this method is demonstrated using a series of PDE5 inhibitors considered to have similar shape. The method shows promise in its capability to handle different conformers, and compares well to existing shape similarity methods. The code and data used to produce the results are available at: \url{https://github.com/RPirie96/KQMolSA}.

\end{abstract}


\section{Introduction and Summary of Part I}

The concept that shared biological activity exists between similar molecules is used widely in drug discovery \cite{Johnson_Maggiora_1990}. Molecules with known activity can be used as templates to screen large databases for other potential hits. This is more efficient and allows coverage of a greater area of chemical space than is possible with experimental screening alone \cite{Leelananda_Lindert_2016}. Estimating similarity between molecules based on their 3D shape has gained popularity due to the requirement for protein-drug shape complementarity to enable strong binding. However no fixed notion of shape exists. Instead, comparison relies on mathematical approximation of the molecule's shape based on its volume, distribution of atomic distances or surface (most commonly treated as the van der Waals or solvent accessible surface) \cite{Kumar_Zhang_2018}. \\
\\
In the accompanying paper \cite{cole2022riemannian}, the RGMolSA method was presented. The descriptor developed there approximates the shape of the molecular surface using a simple nine-element vector containing the surface area and an approximation to the first eight non-zero eigenvalues of the ordinary Laplacian. The descriptor can be viewed as an approximation to the \textit{Riemannian metric}, the underlying mathematical object that describes the shape of a surface. In this paper we present an entirely different method of approximating the Riemannian metric by using ideas from the theory of {\it K\"ahler quantization}; we call this method K\"ahler quantization for Molecular Surface Approximation (KQMolSA). The theory was originally developed by mathematicians and string theorists in order to give explicit representations of the shapes of 4-dimensional objects (Calabi--Yau manifolds) that appear in physical theories (see \cite{DonNRCG} for the paper that pioneered its use as a numerical technique).  In a nutshell, a function called the K\"ahler potential is associated to the metric. We then compute something analogous to a Taylor expansion of this function with the coefficients being stored in a Hermitian matrix.  While the matrices themselves do depend upon the precise position and parameterisation of the molecular surface in three-dimensional space $\mathbb{R}^{3}$, the dependence is easy to calculate.  Hence we can perform our calculations in the `quantized' space of Hermitian matrices and assign a distance between the shapes of two molecular surfaces this way. The final distance is independent of the position of the molecules and the choices made in their parameterisations.

\subsection{Summary of Previous Work}
As in the accompanying paper \cite{cole2022riemannian}, our approach begins by treating the molecule as a series of intersecting spheres, with their radii given by the van der Waals radii of the constituent atoms. The surface is assumed to have a genus of zero, so any rings (e.g. benzene) are replaced with a single sphere of radius 2.25 \AA~to facilitate this. The molecular structure is then defined by the number of spheres $N$ (with each ring counted as a single sphere, and excluding any hydrogen atoms), the centres $c_i$ and radii $r_i$ for each sphere and the adjacency matrix $T$ describing intersection of spheres, where 
\[
T_{ij}=\left\{\begin{array}{cc} 
1 & \textrm{if spheres {\it i} and {\it j} intersect}\\
0 & \textrm{otherwise (or }i=j) .
\end{array}\right.
\]

The surface area $\mathcal{A}$ of the molecule is calculated as the area of each sphere minus the ``missing parts" where two spheres intersect:
\begin{equation}\label{eqn:surf_area}
\mathcal{A} = 2\pi \sum_{i}\left(2r_{i}^{2}-\left(r_{i}\sum_{j}T_{ij}|r_{i}-\lambda_{ij}|\right)\right).
\end{equation}
This value is used to re-scale each of the starting constructs such that the surface area of the molecule is equal to that of a unit sphere (or $4\pi$) to address the observation that Riemannian geometry treats two objects which differ only in size as having equivalent shape. This re-scaling is accounted for in the final descriptors with some weighting so as not to dominate the similarity calculation. \\
\\
From the initial data, a map is constructed to `unwrap` the surface onto the complex plane $\mathbb{C}$ in a process we refer to as piecewise stereographic projection. This requires an atom to be selected as a starting point from which to construct our map, which we refer to as the base sphere. This is taken to be the atom closest to the centre of mass by first finding the centroid of the molecule and then taking the atom with the smallest Euclidean distance from this point. The \textit{Riemannian metric} $g=\Phi_{ps}^{\ast}(g_{Euc})$ induced by the mapping $\Phi_{ps}:\mathbb{C}\rightarrow \mathcal{S}\subset \mathbb{R}^{3}$ takes the form
\begin{equation}\label{eqn:metric_form_gen}
g = \left\{\begin{array}{cc}
\frac{4r_{B}^{2}}{(1+|z|^{2})^{2}}(dx^{2}+dy^{2}) & \mathrm{if} \  z \in\mathcal{C}\\ 
& \\
\frac{C_{1}}{(|z-A_{1}|^{2}+B_{1})^{2}}(dx^{2}+dy^{2}) & \mathrm{if} \  z \in \mathbb{D}(a_{1},R_{1}),\\
 & \\
\frac{C_{2}}{(|z-A_{2}|^{2}+B_{2})^{2}}(dx^{2}+dy^{2}) & \mathrm{if} \  z \in \mathbb{D}(a_{2},R_{2}),\\
\vdots & \vdots \\
\frac{C_{N-1}}{(|z-A_{N-1}|^{2}+B_{N-1})^{2}}(dx^{2}+dy^{2}) & \mathrm{if} \  z \in \mathbb{D}(a_{N-1},R_{N-1}),\\
\end{array}\right.
\end{equation}
where $r_{B}$ is the radius of the base sphere and 
\[
\mathcal{C} = \mathbb{C}\backslash \mathbb{D}(a_{1},R_{1})\cup \mathbb{D}(a_{2},R_{2})\cup \cdots \cup \mathbb{D}(a_{N-1},R_{N-1}),
\]
is the complement of the discs $\mathbb{D}(a_{1},R_{1}), \ldots, \mathbb{D}(a_{N-1},R_{N-1})$ which corresponds to the points in the base sphere. \\
\\
The RGMolSA descriptor uses the explicit form of the Riemannian metric provided by piecewise stereographic projection to approximate the low-lying eigenfunctions of the Laplacian $\Delta$. In \cite{cole2022riemannian}, we compared the RGMolSA descriptor for Sildenafil, Vardenafil and Tadalafil, a series of PDE5 inhibitors that are known to occupy a similar volume in the binding pocket of their target protein, and thus have similar shape (Figure \ref{fig:SVTstruct}) \cite{Cleves_Jain_2008}. Vardenafil is a classic example of a ``me-too" drug, where only a few small modifications have been made to the structure of Sildenafil. As these are both highly similar chemically, they would be expected to have close to the same shape. Tadalafil on the other hand is chemically quite different from the other two, but inspection of the molecules in the pocket of PDE5 reveals they occupy a similar binding pose, and thus would also be expected to have similar shape.  In this article, for ease of comparison with the previous work \cite{cole2022riemannian}, we again use these three molecules as the basis for investigating the new shape descriptor. \\

\begin{figure}
\centering
\subcaptionbox*{}{\includegraphics[width=0.30\textwidth]{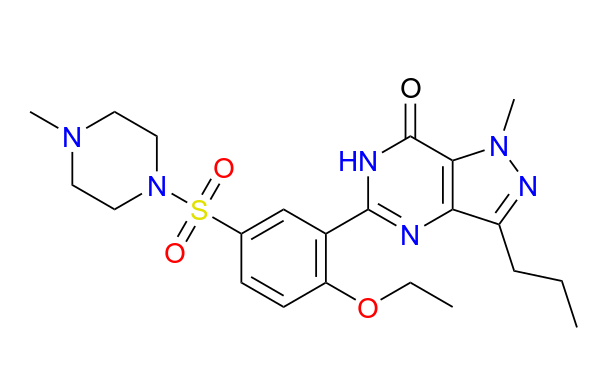}}%
\hfill
\subcaptionbox*{}{\includegraphics[width=0.30\textwidth]{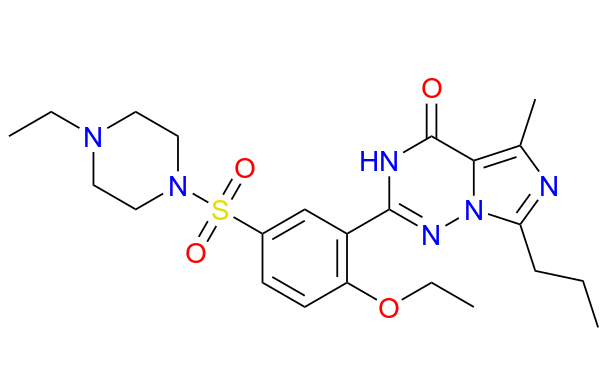}}%
\hfill
\subcaptionbox*{}{\includegraphics[width=0.30\textwidth]{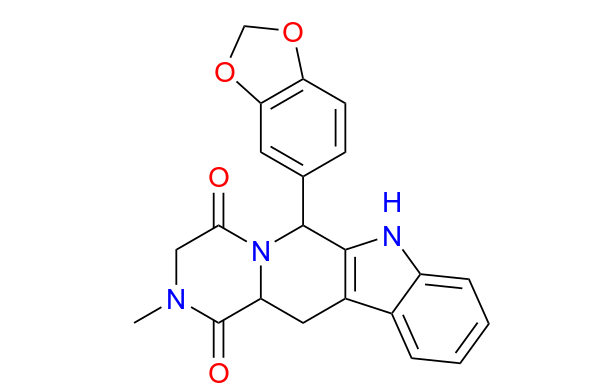}}%
\\
\vspace{-10pt}
\subcaptionbox{Sildenafil \\ Pfizer \\ First Sold: 1998}{\includegraphics[width=0.30\textwidth]{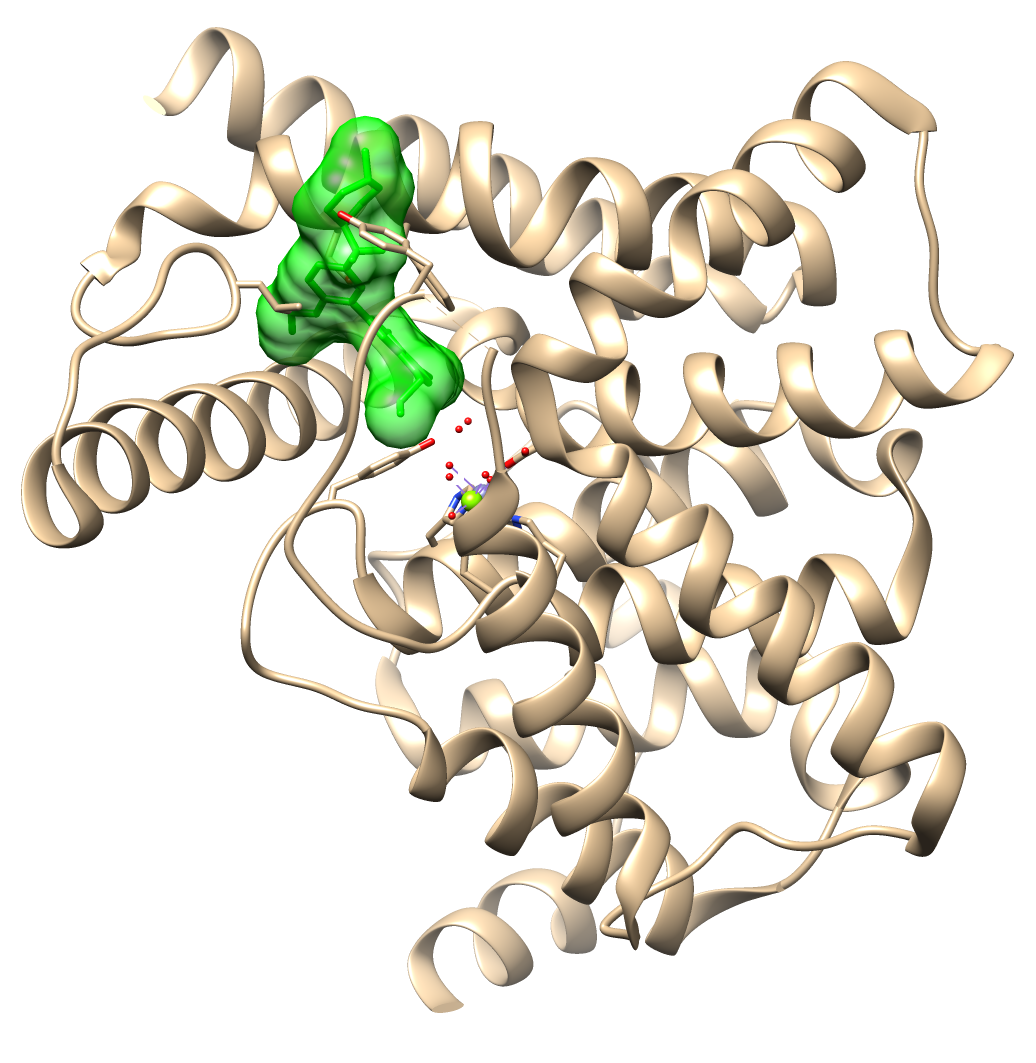}}%
\hfill
\subcaptionbox{Vardenafil \\ Bayer \\ First Sold: 2003}{\includegraphics[width=0.30\textwidth]{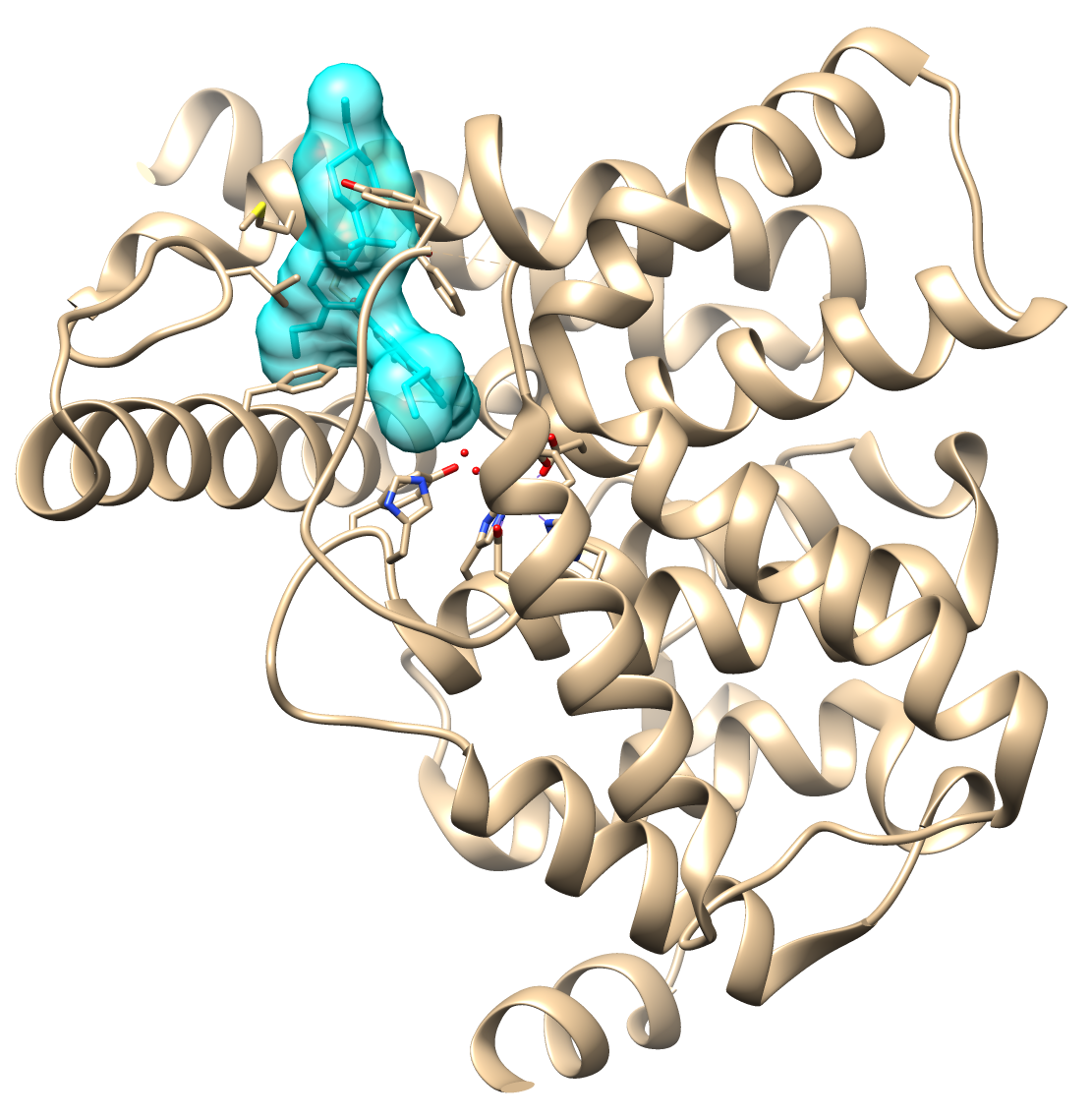}}%
\hfill
\subcaptionbox{Tadalafil \\ Lilly \\ First Sold: 2003}{\includegraphics[width=0.30\textwidth]{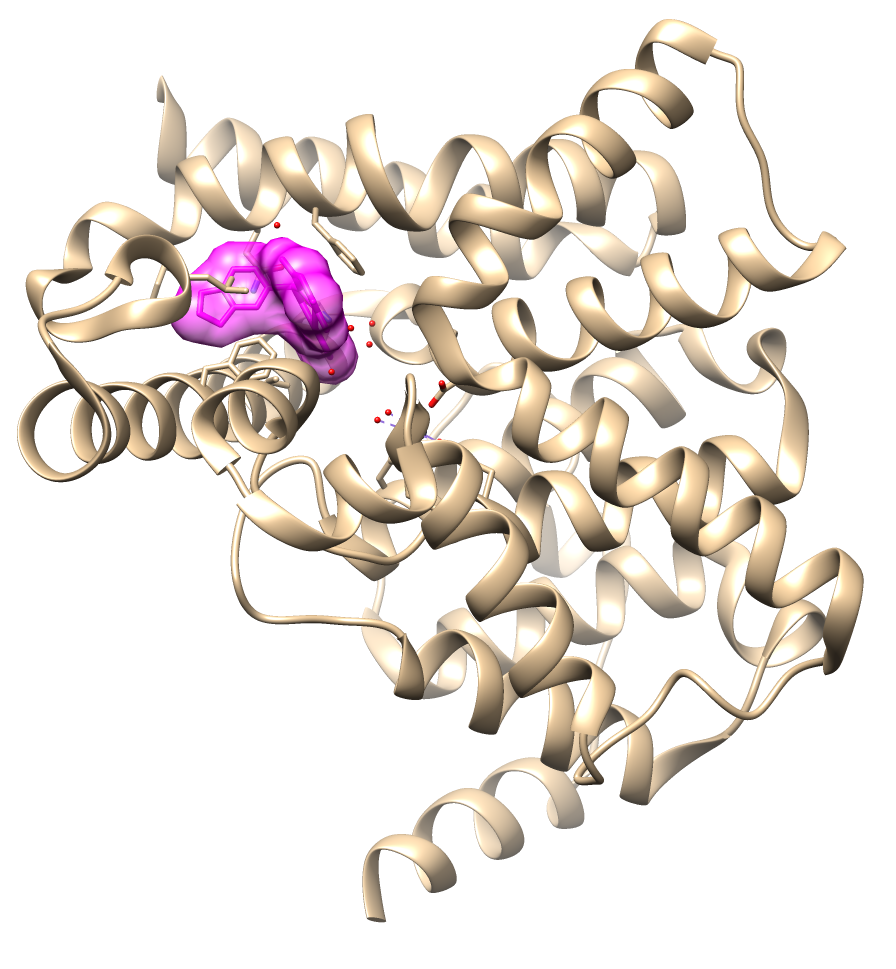}}%
\caption{PDE5 inhibitors Sildenafil, Vardenafil and Tadalafil of known shape similarity. Tadalafil (different chemical structure, similar shape) is an example of a scaffold hop from the first in class drug Sildenafil, and offers greatly improved performance, while Vardenafil (a "me-too" follow-up drug) only offers minor improvements.} \label{fig:SVTstruct}
\end{figure}

While RGMolSA was found to give a good description of shape, it has a possible deficiency due to the dependence of the results on the choice of base sphere, which in turn determines the trial functions for calculating the integrals used to construct the descriptor. The geometry of the surface near the base sphere is well described, but for atoms further away a greater number of eigenvalues would be needed for accurate description of the surface. This problem is greater for larger molecules and can lead to the introduction of numerical errors when the molecule is large enough. We handled such errors by ignoring any contributions from regions with numerical radii less than $10^{-9}$; however, this forces a somewhat artificial `locality' upon the shape descriptor meaning that it probably only accurately captures the shape near to the base sphere. 

In the following section we outline the theory underpinning the KQMolSA descriptors, that again uses the {Riemannian metric} to approximate the geometry of the surface. The resulting descriptors lie in the manifold $GL(N,\mathbb{C})/U(N)$ to give a global descriptor of molecular geometry with reduced dependence on the starting position. Figure \ref{fig:kqflow} summarises the steps in computing these, using Sildenafil as an example. While the descriptor itself does depend upon the choices made and the position of the surface within $\mathbb{R}^{3}$, this is easily accounted for within the space $GL(N,\mathbb{C})/U(N)$.  This makes computing the `distance' between the shape descriptors particularly straightforward. 


\begin{figure}
    \centering
    \includegraphics[scale=0.5]{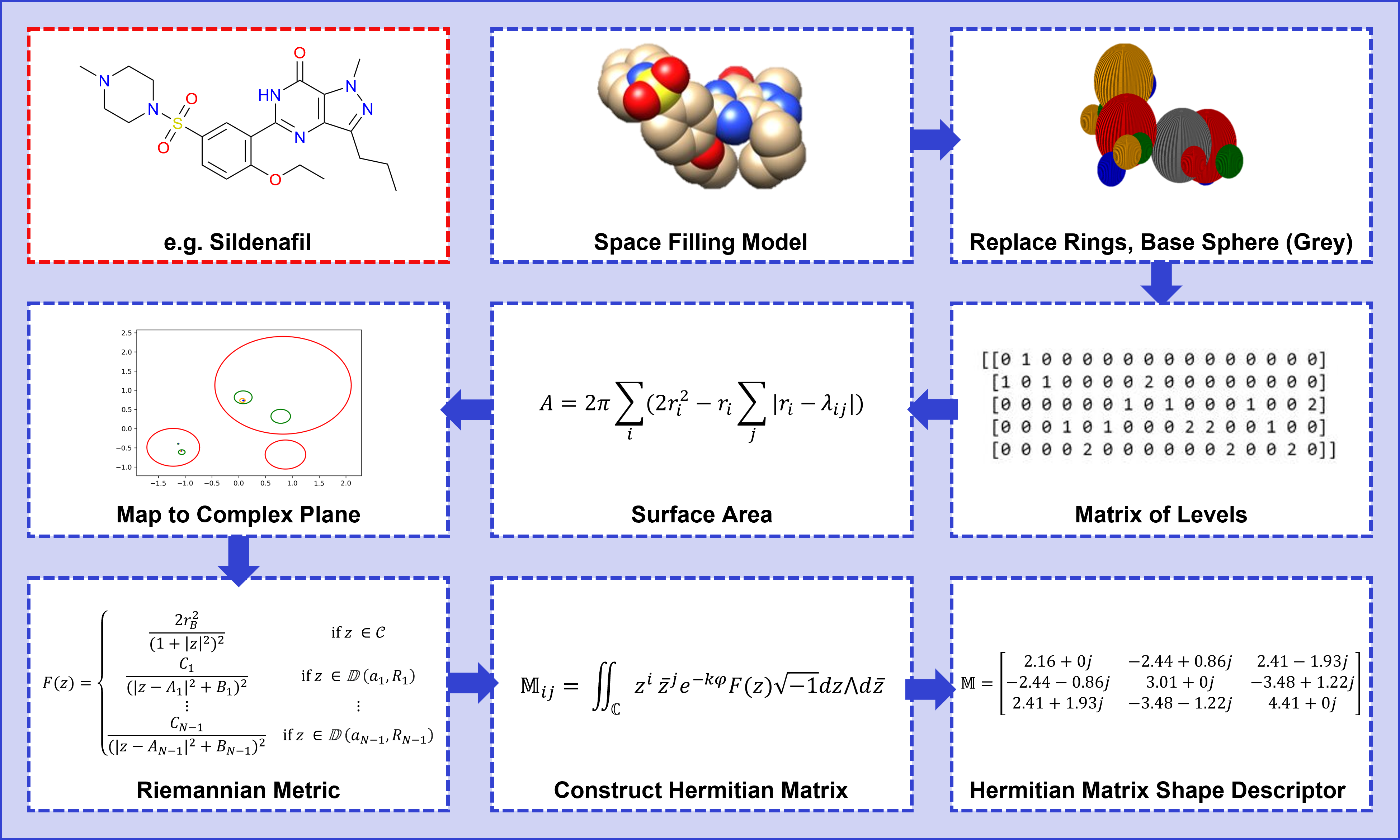}
    \caption{Key steps involved in the computation of the KQMolSA surface descriptor for Sildenafil (a PDE5 inhibitor).}
    \label{fig:kqflow}
\end{figure}

\section{The Mathematics of K\"ahler Quantization} \label{sect:2}
\subsection{Overview of the Theory}
We should say immediately that the theory of K\"ahler quantization is far too advanced to be able to detail in the current paper.  For readers with sufficient mathematical background, a good account (and the original account of its use as a numerical technique) is given in \cite{DonNRCG}. An exposition, aimed at readers with a general scientific background, of the mathematical theory is currently being written by two of the authors \cite{HMbook}.\\ 
\\
The theory is concerned with the geometry of {\it complex manifolds} (shapes that locally look like $\mathbb{C}^{n}$); any surface that sits in $\mathbb{R}^{3}$ is a complex manifold as it locally looks like a copy of the complex numbers $\mathbb{C}$ (i.e. $n=1$). More concretely, we will be concerned with the surfaces that are topologically equivalent to $\mathbb{S}^{2}$; in the language of complex manifolds, the sphere is often referred to as the {\it Riemann Sphere} and denoted $\mathbb{CP}^{1}$. The restriction on the topology of the surface is justified by the fact that chemists do not expect any activity in the centre of rings occurring in most drug-like molecules. The exceptions to this are macrocyclic molecules (those with large rings of more than 12 atoms) where genuine activity occurs in the centre of the ring. Such molecules are therefore excluded from comparison by both methods proposed. \\
\\
The natural class of functions to work with when dealing with complex manifolds are those that are complex differentiable, often called {\it holomorphic} functions. We consider a general complex manifold $X$; unfortunately, if the manifold $X$ is compact, the only holomorphic functions $f:X\rightarrow \mathbb{C}$ are constant. Thus we cannot hope to understand $X$ simply by studying the holomorphic functions on $X$. A generalisation of the notion of a holomorphic function is that of a section of a holomorphic line bundle $L$ with base $X$. For readers familiar with the theory, a function is a section of the trivial bundle. A line bundle is positive if there is a Hermitian metric $h$ on $L$ with positive curvature. A foundational result of Kodaira \cite{GrifHarr} says that if the line bundle $L$ is positive then for large enough $k$ the tensor power $L^{k}$, has a lot of holomorphic sections.  In fact, the space of all such sections, denoted $H^{0}(L^{k})$, is a complex vector space of dimension that has order $O(k^{n})$ as $k\rightarrow \infty$. \\
\\
The curvature of a positively curved Hermitian metric $h$ gives rise to an object called a {\it K\"ahler form}, $\omega$, which in turn gives rise to a Riemannian metric $g$ (the mathematical object being used in \cite{cole2022riemannian} to describe shape). It turns out that the set of all positively curved Hermitian metrics on a line bundle $L$ can be identified with the set of all real-valued functions $\varphi:X \rightarrow \mathbb{R}$ that satisfy, in some local coordinate $z$, the $\partial\bar{\partial}$-equation
\[
\sqrt{-1}\partial\bar{\partial} \varphi =\omega-\omega_{0}
\]
where $\omega$ is the K\"ahler form of the metric and $\omega_{0}$ is a fixed reference K\"ahler form. We will give more detail on the differential operators $\partial$ and $\bar{\partial}$ in Section \ref{subsec:IinP}; in particular, we will explain that in the molecular surface setting, the $\partial\bar{\partial}$-equation is really just the familiar Poisson equation in the plane. The function $\varphi$ is called a {\it K\"ahler potential} for $\omega$. The associated potential is not unique but any two differ by a constant; this does not affect the metric which is constructed by taking two derivatives of the potential. However, we will see that the addition of a constant to a potential will have the affect of scaling the Hermitian matrix we produce as a shape descriptor by a positive real number and we will be required to find the `optimal' rescaling in our distance calculation.\\
\\
To summarise, what we have for a positive Hermitian line bundle $(L,h)\rightarrow X$ are:
\begin{itemize}
    \item   a K\"ahler form $\omega$ and a K\"ahler potential $\varphi:X\rightarrow \mathbb{R}$,
    \item a complex vector space $H^{0}(L^{k})$.
\end{itemize}
What K\"ahler quantization amounts to is relating the geometry described by the K\"ahler potentials (an infinite dimensional space of functions) to the finite dimensional complex vector space $H^{0}(L^{k})$. This theme occurs throughout numerical analysis and shape description, for example in the theories of Fourier analysis, spherical harmonics, Taylor series, all of which produce a finite-dimensional vector space out of some infinite-dimensional set of functions.
\subsection{Quantization and Tian's Theorem}
The data $(L,h)\rightarrow X$ allows for a natural $\mathcal{L}^{2}$-inner product on the vector space of sections $H^{0}(L^{k})$.  Given sections $s_{1},s_{2}\in H^{0}(L^{k})$, we compute
\[
\langle s_{1},s_{2}\rangle := \int_{X}h_{k}(s_{1},s_{2})\dfrac{\omega^{n}}{n!},  
\]
where $h_{k}$ is the Hermitian metric induced on $L^{k}$ by $h$, and ${\omega^{n}}/{n!}$ is the volume element produced by the K\"ahler form. It is this inner product that is the quantization of the data $(L,h)\rightarrow X$. The space of all (Hermitian) inner products on a complex $N$-dimensional vector space can be thought of as $GL(N;\mathbb{C})/U(N)$. This is a negatively curved symmetric space and has a natural notion of distance on it; it is this distance that we will use to measure shape similarity (see Section \ref{section:d_measure}).\\
\\
To recover the geometry defined by $(L,h)\rightarrow X$ from the quantization, we choose a basis $\{s_{j}\}$ of the vector space $H^{0}(L^{k})$ which gives rise to the matrix representation of the inner product
\[
\mathbb{M}_{ij} :=\langle s_{i},s_{j}\rangle.
\]
If we let $v$ be the vector of sections
\[
v=\left(s_{1},s_{2},\ldots s_{N}\right),
\]
then we can define a K\"ahler potential (recalling that the sections are locally defined holomorphic functions) $\tilde{\varphi}$ by
\[
\tilde{\varphi}:=-\dfrac{1}{k}\log\left(v^{\ast}\mathbb{M}^{-1}{v}\right). 
\]
\begin{theorem}[Tian, \cite{Tian}]
Let $(X,L,h)$ be a complex manifold with holomorphic line bundle $L$ and positively curved Hermitian metric $h$ with curvature $\omega$. If we produce another K\"ahler form 
\[
\widetilde{\omega} = \omega_{0}+\sqrt{-1}\partial\bar{\partial}\tilde{\varphi},
\]
then
\[
\|\omega-\widetilde{\omega}\|_{C^{0}} = O(k^{-2}).
\]
\end{theorem}
Paraphrasing this theorem, we can say any K\"ahler form coming from a K\"ahler potential $\varphi$ can be well approximated by the K\"ahler form coming from the `algebraic' function $\widetilde{\varphi}$.  If we pick local complex coordinates $z_{1},z_{2},\ldots, z_{n}$ then the term $v^{\ast}\mathbb{M}^{-1}v$ is just a power series in the coordinates. In the case of 
 a molecular surface, we will have something like a polynomial.  This is the sense in which the function $\widetilde{\varphi}$ is similar to a truncated Taylor series for the original function $\varphi$. The theorem then says that this series really does converge.\\
\\
Tian's Theorem is stated for smooth metrics (those where one can take an arbitrary number of derivatives of the K\"ahler potential $\varphi$); in practice (see Section \ref{subsec:IinP}), we will be working with metrics where the potentials are in $\mathrm{C}^{2}(X)$, that is twice continuously differentiable.  The theory of approximating such metrics algebraically has not been written down but we will demonstrate that we get a method that does produce meaningful shape comparisons. We expect that, suitably adapted to this setting, something like Tian's Theorem is still true; for example, the case of potentials with lower regularity is discussed in \cite{BerKel}.
\subsection{Implementation in Practice} \label{subsec:IinP}
As mentioned already, in practice we take $X=\mathbb{CP}^{1}$ the Riemann  sphere and the line bundle to be the anticanonical bundle $K^{\ast}_{\mathbb{CP}^{1}}=\mathcal{O}(2)$. The K\"ahler form $\omega$, can be explicitly constructed from the Riemannian metric $g$, and in the coordinates furnished by the piecewise stereographic projection map $\Phi_{ps}$, we can use the form of the metric  (\ref{eqn:metric_form_gen}) to get 
\[
\omega = F(z)\sqrt{-1}dz\wedge d\overline{z},
\]
where $F:\mathbb{C}\rightarrow \mathbb{R}_{+}$ is the `metric function' given by
\begin{equation}\label{eqn:form_of_KF}
F(z) = \left\{\begin{array}{cc}
\frac{2r_{B}^{2}}{(1+|z|^{2})^{2}} & \mathrm{if} \  z \in\mathcal{C},\\ 
& \\
\frac{C_{1}}{(|z-A_{1}|^{2}+B_{1})^{2}} & \mathrm{if} \  z \in \mathbb{D}(a_{1},R_{1}),\\
 & \\
\frac{C_{2}}{(|z-A_{2}|^{2}+B_{2})^{2}} & \mathrm{if} \  z \in \mathbb{D}(a_{2},R_{2}),\\
\vdots & \vdots \\
\frac{C_{N-1}}{(|z-A_{N-1}|^{2}+B_{N-1})^{2}} & \mathrm{if} \  z \in \mathbb{D}(a_{N-1},R_{N-1}).\\
\end{array}\right.
\end{equation}
Note we have replaced, in the metric $g$, the real symmetric 2-tensor $dx^{2}+dy^{2}$ with the antisymmetric form ${(\sqrt{-1}/2)dz\wedge d\bar{z}}$, where ${dz=dx+\sqrt{-1}dy}$ and ${d\bar{z}=dx-\sqrt{-1}dy}$.

To find the K\"ahler potential $\varphi:\mathbb{C} \rightarrow \mathbb{R}$, we solve the `$\partial\overline{\partial}$-equation'
\[
\omega = \sqrt{-1}\partial\overline{\partial}\varphi.
\]
If we consider the complex differential operators
\[
\dfrac{\partial}{\partial z} = \dfrac{1}{2}\left(\dfrac{\partial}{\partial x}-\sqrt{-1}\dfrac{\partial}{\partial y}\right) \qquad \mathrm{and} \qquad \dfrac{\partial}{\partial \overline{z}} = \dfrac{1}{2}\left(\dfrac{\partial}{\partial x}+\sqrt{-1}\dfrac{\partial}{\partial y}\right),
\]
then the $\partial\overline{\partial}$-equation is equivalent to solving the Poisson equation
\[
\frac{\partial^{2}\varphi}{\partial z\partial\overline{z}}=\frac{1}{4}\Delta_{Euc} \varphi = F,
\]
where $\Delta_{Euc}$ is the usual 2-dimensional Laplacian.
We can solve the Poisson problem explicitly to find $\varphi$. The solution can be thought of as having two parts: a `local' part that is found by simply observing that
\[
\frac{\partial^{2}}{\partial z\partial\overline{z}}\left(\frac{C\log(|z-A|^{2}+B)}{B}\right) = \dfrac{C}{(|z-A|^{2}+B)^{2}},
\]
and a `correction term', named thus as the term is needed to ensure the function is in $\mathrm{C}^{2}(\mathbb{C})$. The correction term is a linear combination of functions of the form
\[
\log(|\alpha z+\beta|^{2}),
\]
where we get one term for each sphere. As each of the correction terms is a harmonic function, that is
\[
\Delta \log(|\alpha z+\beta|^{2})=0,
\]
the addition of the correction terms is still a solution of the Poisson equation. It would appear the correction terms are singular at the points $z=-\beta/\alpha$; however, these points always lie outside the disc where the function takes this particular form. We record the form of the potential as a theorem and refer the reader to the appendix (Section~\ref{sec:appendix}) for a derivation of the solution.
\begin{theorem}[Form of K\"ahler potential]
Let $g$ be of the form Equation~(\ref{eqn:metric_form_gen}). In the region associated to the $i^{th}$ sphere, the K\"ahler potential can be written
\[
\varphi(z) = \frac{C_{i}}{B_{i}}\log(|z-A_{i}|^2+B_{i})+\sum_{j=1}^{N}\mathcal{K}_{ij}\log(|\alpha_{ij}z+\beta_{ij}|^{2}),
\]
where $\mathcal{K} \in M^{N\times N}(\mathbb{R})$, and $\alpha, \beta \in M^{N\times N}(\mathbb{C})$.
\end{theorem}
The matrices $\mathcal{K}, \alpha,$ and  $\beta$  in the previous theorem are easily calculated from the geometric data associated to the molecule and so it is straightforward to describe the K\"ahler potential explicitly.\\ 
\\
The space of global sections $H^{0}(\mathcal{O}(2k)) \cong \mathbb{C}^{2k+1}$ can be identified with the span of the functions
\[
\langle 1,z,z^{2},\ldots,z^{2k} \rangle.
\]
Thus the shape descriptor associated to the surface is the $(2k+1)\times(2k+1)$ Hermitian matrix $\mathbb{M}$ where (considering indices that run from 0 to $2k$)
\begin{equation} \label{eqn:qm_entry}
\mathbb{M}_{ij} =  \iint_{\mathbb{C}}z^{i}\overline{z}^{j} e^{-k\varphi}F(z)\sqrt{-1}dz \wedge d\overline{z}.    
\end{equation}

\subsection{Computing the Relevant Integrals}
\label{sec2.4}
A na\"ive numerical calculation of the integrals described by Equation (\ref{eqn:qm_entry}) gives rise to two obvious problems: firstly, the domain of integration is unbounded (being the whole complex plane $\mathbb{C}$); secondly, the domains and values describing the metric and the K\"ahler potential $\varphi$ could become so small that numerical instabilities start to dominate the contribution of the associated atom. The second problem has been discussed as a limitation in the approximation of the spectrum of the Laplacian \cite{cole2022riemannian}. In this paper, we exploit the fact that the automorphism group of $\mathbb{CP}^{1}$ is the group of M\"obius transformations, $PSL(2,\mathbb{C})$; we can use elements of this group to ensure the coordinates we perform calculations in are always in a numerically controlled region (here we use a unit disc).\\
\\
Put more concretely, let $m\in \{1,2,\ldots,N\}$ index the $m^{th}$ sphere making up the molecular surface, then there is an element $\mathcal{T}_{m}\in PSL(2,\mathbb{C})$ that maps the unit disc
\[
\mathbb{D} =\{z \in \mathbb{C} \ | \ |z|<1\},
\]
onto the region $\mathbb{D}(a_{m},R_{m})$ from Equation (\ref{eqn:metric_form_gen}). We note that if the $m^{th}$ sphere has level $l$, then the pre-image of the regions corresponding to level $(l+1)$ spheres which intersect the $m^{th}$ sphere will describe certain discs properly contained in $\mathbb{D}$.  Hence the contribution of the $m^{th}$ sphere to the matrix described by Equation (\ref{eqn:qm_entry}) is given by
\begin{equation}\label{eqn:disc_integral}
\iint_{\mathbb{D}-\hat{D}}(\mathcal{T}_{m}(w))^{i}(\overline{\mathcal{T}_{m}(w)})^{j}e^{-k\varphi(\mathcal{T}_{m}(w))}F(\mathcal{T}_{m}(w)) \ d\mathcal{T}_{m}(w)\wedge d\overline{\mathcal{T}_{m}(w)},
\end{equation}
where $\hat{D}$ represents the union of the discs corresponding to the next level spheres intersecting the $m^{th}$ sphere. In practice, we account for these higher-level spheres by assigning the value $0$ to the volume form $F(\mathcal{T}_{m}(w)) \ d\mathcal{T}_{m}(w)\wedge d\overline{\mathcal{T}_{m}(w)}$ whenever $w \in \hat{D}$ (note this produces a jump discontinuity in the volume form). Numerical calculation of integrals of the form of Equation~(\ref{eqn:disc_integral}) is done by splitting into an angular and radial direction and then performing successive applications of the trapezium rule; we choose a radial step size corresponding to $n_{r}=15$ integration points and an angular step size corresponding to taking $n_{\theta}=10$ points. This seems to achieve a reasonable accuracy; for example, one can check the area integral for a given integration scheme. We have also determined that the distance between shape descriptors does not seem to be significantly changed by taking smaller step sizes (Section \ref{section:drdt}).

\subsection{Finding the Distance Between Shape Descriptors} \label{section:d_measure}
Given two positive definite Hermitian matrices $\mathbb{M}_{1}, \mathbb{M}_{2}$, such as those generated by Equation (\ref{eqn:qm_entry}), there are innumerable ways of defining a notion of distance between such matrices.  With regards to the theory of K\"ahler quantization, it is natural to consider  $\mathbb{M}_{1}, \mathbb{M}_{2}$ as two Hermitian inner products on the fixed complex vector space $H^{0}(\mathcal{O}(2k))$. This space is naturally seen as the manifold $GL(2k+1;\mathbb{C})/U(2k+1)$. An inner product is specified by declaring a particular basis to be orthonormal; any basis conjugate under the action of $U(2k+1)$ defines the same inner product.  This space has a natural distance on it; one characterisation of this distance is that shortest paths (geodesics) are given by one-parameter subgroups of $GL(2k+1;\mathbb{C})$, that is by paths of matrices of the form $\mathrm{exp}(tA)$ where $A$ is some $(2k+1)\times(2k+1)$ complex matrix.\\
\\
More explicitly, if $\{v_{1},v_{2},\ldots,v_{2k+1} \}$ is a basis of $H^{0}(\mathcal{O}(2k))$such that both inner products are represented by diagonal matrices
\[
\mathbb{M}_{1} = \mathrm{Diag}\left(e^{\lambda_{1}},e^{\lambda_{2}},\ldots,e^{\lambda_{2k+1}}\right), \qquad \mathbb{M}_{2} = \mathrm{Diag}\left(e^{\mu_{1}},e^{\mu_{2}},\ldots,e^{\mu_{2k+1}}\right),
\]
then
\begin{equation}\label{eqn:distance}
d(\mathbb{M}_{1},\mathbb{M}_{2})=k^{-\frac{3}{2}}\sqrt{\sum_{i=1}^{2k+1}(\lambda_{i}-\mu_{i})^{2}}.    
\end{equation}
The factor of $k^{-3/2}$ ensures that the distances stabilise as $k\rightarrow \infty$ (see Theorem 1.1 in \cite{ChenSun}). It will be useful to consider the following more compact form for the distance
\begin{equation}\label{eqn:distance_compact}
d(\mathbb{M}_{1},\mathbb{M}_{2})=k^{-\frac{3}{2}}\sqrt{\sum_{i=1}^{2k+1}(\log(\eta_{i}))^{2}},    
\end{equation}
where $\{{\eta_{i}\}}$ are the eigenvalues of the matrix $\mathbb{M}_{1}^{-1}\mathbb{M}_{2}$. \\
\\
It is a well-known fact that the automorphism group of the Riemann sphere $\mathbb{CP}^{1}$ is the group of M\"obius transformations $PSL(2,\mathbb{C})$.  Roughly speaking, the subgroup $PSU(2)\subset PSL(2,\mathbb{C})$ corresponds to rotations of the original surface and the remaining maps correspond to reparameterisations that preserve the complex structure. If $\varpi \in PSL(2,\mathbb{C})$ is an automorphism of the form
\[
\varpi (z) = \frac{\alpha z+\beta}{\gamma z+ \delta},
\]
then $\varpi$ also acts on the vector space $H^{0}(\mathcal{O}(2k))$. In representation theoretic terms, this action is the representation induced on $\mathrm{Sym}_{2k}(\mathbb{C}^{2})$ by the standard representation of $SL(2,\mathbb{C})$. If we denote the element of $SL(2k+1,\mathbb{C})$ by $\vartheta(\varpi)$  (see \cite{Hash}, Lemma 8) and the original shape descriptor computed in the $z$-coordinate by $\mathbb{M}$, then the shape descriptor computed in the $\varpi (z)$-coordinate will be 
\[ \left(\vartheta(\varpi)\right)^{\ast}\mathbb{M}\left(\vartheta(\varpi)\right).
\]

As mentioned in Section \ref{sect:2}, the fact that the K\"ahler potential is only defined up to the addition of a constant means we can also scale the Hermitian matrix $\mathbb{M}$ by a positive constant. Hence our calculation of distance between two shape descriptors $\mathbb{M}_{1}$ and $\mathbb{M}_{2}$ becomes the concrete problem of minimising, over $(p,\vartheta) \in \mathbb{R}\times SL(2,\mathbb{C})$, 
\[
\zeta(p,\vartheta) = \sum_{i=1}^{2k+1}(\log(\eta_{i}))^{2},
\]
where $\{{\eta_{i}\}}$ are the eigenvalues of the matrix $\mathbb{M}_{1}^{-1}e^{p}\left(\vartheta(\varpi)\right)^{\ast}\mathbb{M}_{2}\left(\vartheta(\varpi)\right)$.\\
\\
It is easy to see that the value of $p$ at a critical point of $\zeta$ is independent of the element $\vartheta$. Elementary calculus yields that the value of $p$ is given by
\[
p=-\frac{1}{2k+1}\sum_{i=1}^{2k+1}\log(\tilde{\eta}_{i}),
\]
where $\{\tilde{\eta}_{i}\}$ are the eigenvalues of the matrix $\mathbb{M}_{1}^{-1}\mathbb{M}_{2}$. As the matrix $\left(\vartheta(\varpi)\right)$ has unit determinant, the value of $p$ does not depend up the $SL(2,\mathbb{C})$ action on the Hermitian matrix $\mathbb{M}_{2}$. We thus reduce the distance calculation to a minimisation over the six-dimensional Lie group $SL(2,\mathbb{C})$.\\
\\
Note that the distance between the shape descriptors given by Equation~(\ref{eqn:distance}) is the distance between the molecular shapes {\it after} they have been re-scaled to have area $4\pi$. Hence the distance between two molecular surfaces $\mathcal{S}_{1}$ and $\mathcal{S}_{2}$ should include a component to reflect the difference in area between $\mathcal{S}_{1}$ and $\mathcal{S}_{2}$. As we are interested in producing a similarity score rather than a distance between two inputs, we do not take this point up further in the article. Our initial attempts at creating a similarity score are detailed in the subsequent section.\\
\\
The remaining minimisation over $SL(2,\mathbb{C})$ is done by parameterising a generic matrix by the $6$ real variables $x_{1},...x_{6}$ and taking
\[
\varpi(x_{1},x_{2},\ldots,x_{6}) = \left(\begin{array}{cc}
x_{1}+\sqrt{-1}x_{2} & x_{3}+\sqrt{-1}x_{4}
\\
x_{5}+\sqrt{-1}x_{6} & \ast
\end{array}\right),
\]
where $\ast$ is chosen to ensure $\det(\varpi)=1$. To perform the minimisation, we use algorithms that do not require the input of a gradient vector, such as  Nelder--Mead or Powell methods \cite{NumRec}.  These are implemented using off-the-shelf packages in SciPy \cite{2020SciPy-NMeth}.  We found that for $k=1$ there was very little difference between the results for either method; the minimisation algorithm converges to produce a robust distance value. For $k=2$ the minimisation methods appear to be a little less stable and occasionally did not converge. One way around this was to use the element of $SL(2,\mathbb{C})$ found by the $k=1$ minimisation as the initial guess for the $k=2$ step (otherwise the identity matrix was used). We anticipate that one might be able to improve this process; for example, by computing the gradient of the function to be minimised explicitly and then using this in an algorithm such as conjugate gradient descent.\\
\\
One further consideration in implementing the distance measure between two matrices was in shape descriptors for $k>2$ (and for $k=2$ in some cases), where numerical instability exists within the method. Occasionally non-positive definite matrices are produced, that cannot be compared using the above approach. As Hermitian matrices that differ only by scale can be considered equivalent, such cases have been treated by scaling one matrix by a factor of 10, 100 or 1000 as needed in order to bring the eigenvalues into the range required for consideration with Python. 

\section{Initial Case Study: Phosphodiesterase 5 (PDE5) Inhibitors}
\label{sec:4}

\subsection{Tuning the Parameters $\mathbf{n_{r}}$ and $\mathbf{n_{\theta}}$}\label{section:drdt}

To determine the effect of varying the parameters $n_r$ and $n_\theta$ (Section~\ref{sec2.4}) on the quality of the shape descriptors produced, we considered three sets of parameters: $n_r=200$ and $n_\theta=100$; $n_r=50$ and $n_\theta=25$; $n_r=15$ and $n_\theta=10$. The distances produced between the descriptor for each set and the area returned during the computation of the relevant integrals (which should be $\sim12.57$ for an accurate descriptor, as constrained by the choice of scaling the surface area to $4\pi$) are reported here for Sildenafil (Table \ref{tab:tunedrdt_S}), Vardenafil (Table \ref{tab:tunedrdt_V}) and Tadalafil (Table \ref{tab:tunedrdt_T}). 

\begin{center}
\begin{table}[h]
\centering
\caption{Computed distances between descriptors of Sildenafil generated using different values of $n_r$ and $n_\theta$ for $k=1$. The area reported is that returned by the integration step.}
\begin{tabular}{|c|c|c|c|} 
\hline 
 & \bf (200, 100), area = 12.59 & \bf (50, 25), area = 12.62 & \bf (15, 10), area = 12.62 \\
 \hline
\bf (200, 100) & - & 0.032 & 0.032 \\
\hline
\bf (50, 25) & 0.038 & - & 0.040\\
\hline
\bf (15, 10) & 0.038 & 0.040 & -\\
\hline
\end{tabular}
\label{tab:tunedrdt_S}
\end{table}
\end{center}
\vspace{-30pt}
\begin{center}
\begin{table}[h]
\centering
\caption{Computed distances between descriptors of Vardenafil generated using different values of $n_r$ and $n_\theta$ for $k=1$. The area reported is that returned by the integration step.}
\begin{tabular}{|c|c|c|c|} 
\hline 
 & \bf (200, 100) area = 12.57 & \bf (50, 25), area = 12.58 & \bf (15, 10), area = 12.58 \\
 \hline
\bf (200, 100) & - & 0.005 & 0.005 \\
\hline
\bf (50, 25) & 0.005 & - & 0.004 \\
\hline
\bf (15, 10) & 0.005 & 0.004 & -\\
\hline
\end{tabular}
\label{tab:tunedrdt_V}
\end{table}
\end{center}
\vspace{-30pt}
\begin{center}
\begin{table}[h!]
\centering
\caption{Computed distances between descriptors of Tadalafil generated using different values of $n_r$ and $n_\theta$ for $k=1$. The area reported is that returned by the integration step.}
\begin{tabular}{|c|c|c|c|} 
\hline 
 & \bf (200, 100), area = 14.32 & \bf (50, 25), area = 14.37 & \bf (15, 10), area = 14.37 \\
 \hline
\bf (200, 100) & - & 0.003 & 0.003 \\
\hline
\bf (50, 25) & 0.003 & - & 0.001 \\
\hline
\bf (15, 10) & 0.003 & 0.001 & -\\
\hline
\end{tabular}
\label{tab:tunedrdt_T}
\end{table}
\end{center}
\vspace{-20pt}
As these distances are small in each case, there is no significant loss of quality when the number of points considered is reduced. The areas for both Sildenafil and Vardenafil are also close to 12.57, indicating high quality descriptors. The area for Tadalafil is overestimated slightly, however this is due to an issue with the replacement of the rings for motifs with a 5-membered ring between two other rings rather than the choice of $n_r$ and $n_\theta$. Similar results were observed for the consideration of $k=2$. As the quality is unaffected, the minimum parameters of $n_r=15$ and $n_\theta=10$ were used in the final descriptors to increase the speed of calculation. 

\subsection{Constructing a Similarity Score} \label{section:weightSA}

In order to facilitate familiar comparison of molecules, we wish to construct a similarity score rather than simply taking the distance between two matrices. In chemoinformatics, this score typically takes a value between 0 (no similarity) and 1 (identical) \cite{Kumar_Zhang_2018}. To achieve this we take the inverse distance, and account for size by taking the ratio of two surface areas. Equation \ref{eqn:simscore} gives the similarity score between two molecular surfaces $\mathcal{S}_{1}$ and $\mathcal{S}_{2}$, 

\begin{equation} \label{eqn:simscore}
   score(\mathcal{S}_{1},\mathcal{S}_{2}) = x (A_{min} / A_{max}) + y \frac{1}{1 + d(\mathbb{M}_{1},\mathbb{M}_{2})},
\end{equation}

where $A_{min}$ is the smaller of the two surface areas, and $A_{max}$ is the larger, in order to give a score bounded by 0 and 1. We therefore need to choose an appropriate set of weights $x$ and $y$ such that $x + y = 1$, and $x < 0.5$, to ensure the shape is the primary contributor to the score.



Table \ref{tab:weightSA} gives the resulting similarity scores for pairwise comparison of the PDE5 inhibitors. In all three cases, the similarity increases with increasing contribution from the surface area term as expected. The increase for Sildenafil-Vardenafil is only small, while for Tadalafil there is a greater effect of including the area. Final weights of $x=0.3$ and $y=0.7$ were selected to balance the contribution of the surface area without it dominating over the shape contribution. The PDE5 inhibitors were selected for tuning due to their known similarity, however further refinement of these parameters with a larger set of examples may be required for full scale virtual screening. \\
\\
\begin{center}
\begin{table}[h]
\centering
\caption{Similarity scores for the PDE5 inhibitors for surface area weights ranging from 0 to 0.5.}
\begin{tabular}{|c|c|c|c|c|} 
\hline 
\bf x & \bf y & \bf Sildenafil-Vardenafil & \bf Sildenafil-Tadalafil & \bf Vardenafil-Tadalafil \\
\hline
0 & 1 & 0.884 & 0.286 & 0.275 \\
0.1 & 0.9 & 0.892 & 0.340 & 0.328 \\
0.2 & 0.8 & 0.900 & 0.394 & 0.380 \\
0.3 & 0.7 & 0.908 & 0.449 & 0.432 \\
0.4 & 0.6 & 0.916 & 0.503 & 0.485\\
0.5 & 0.5 & 0.924 & 0.557 & 0.537 \\
\hline
\end{tabular}
\label{tab:weightSA}
\end{table}
\end{center}
\vspace{-40pt}
\subsection{Investigating Variation in 3D Conformers}

As discussed in the previous work \cite{cole2022riemannian}, consideration of the different orientations a molecule can adopt (known as conformers) is important when using 3D shape descriptors. Conformers of the same molecule should theoretically have scores in the range $0.7 < score < 1$, as high self-similarity is expected (scores above $0.7$ in chemoinformatics), while retaining the ability to distinguish between them. \\
\\
As with RGMolSA, two small sets of 10 conformers of the PDE5 inhibitors are used to investigate how KQMolSA regards different conformers. One set contains 10 random conformers, in which we would expect slightly more variance, while the other has 10 low energy conformers, for which higher similarity is expected. Both sets were produced using the ETKDG algorithm \cite{Riniker_Landrum_2015} with energy optimisation using the MMFF94 force field \cite{tosco_stiefl_landrum_2014}, both implemented in RDKit \cite{landrum}. The minimum, maximum and average shape similarity as well as the average RMSD (which compares conformers based on their atomic positions) for each set are given in Figure~\ref{fig:conf_sim}. The full set of RMSD and shape similarity comparisons are available in the {\bf Supporting Data}.

\begin{figure}
	\centering
	\subcaptionbox{$k=1$}{\includegraphics[width=\textwidth]{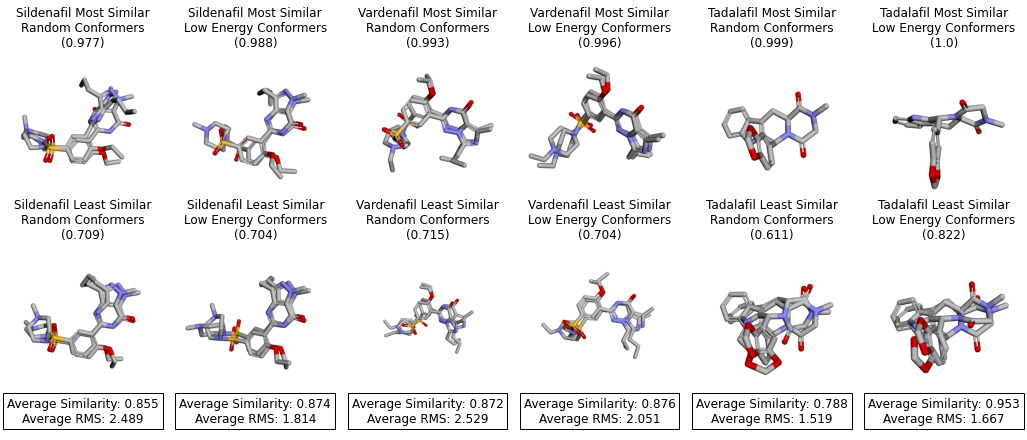}}%
	\\
	\subcaptionbox{$k=2$}{\includegraphics[width=\textwidth]{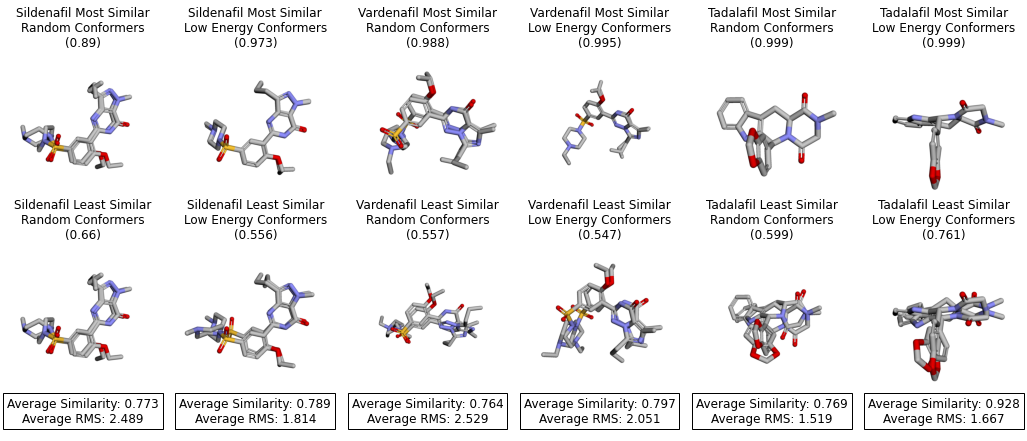}}%
	\caption{Overlay of the most and least shape-similar conformers of Sildenafil, Vardenafil and Tadalafil and the average shape similarity and RMSD for each set for (a) $k=1$ and (b) $k=2$. On average the conformers display a high degree of self-similarity despite the variance in atom-position similarity.} \label{fig:conf_sim}
\end{figure}




The RMSD and shape similarity for each set are compared in the swarm plots shown in Figure~\ref{fig:simvsrms}. For $k=1$, generally high similarity was observed, with some scores for the random conformers of Tadalafil falling slightly below 0.7. Greater variation is observed for $k=2$, where some conformer pairs have scores below 0.6. This reduction in similarity is expected for $k=2$ as the descriptors represent a more detailed approximation to the original surface than those for $k=1$ and hence will be more sensitive to differences in the geometry. However, the similarity scores obtained were on the whole lower than for RGMolSA, where the similarity between most conformer pairs is greater than 0.8 \cite{cole2022riemannian}. For the random sets, the similarity between conformers showed more variation than for RGMolSA, where clusters of similar conformers were observed. While KQMolSA does handle conformers well, RGMolSA appears to do a better job of this, due to the insensitivity to surface deformation of the spectrum of the Laplace–Beltrami operator. For virtual screening, this consideration of conformers as similar negates the need for a pre-alignment step prior to shape similarity calculation, and may allow molecules that can deform to fit in the binding pocket to be identified as potential hits, where these would otherwise be classified as the wrong shape by methods that depend on atomic coordinates. 




\begin{figure}[!h]
	\centering
	\subcaptionbox{RMS Similarity}{\includegraphics[width=7.5cm]{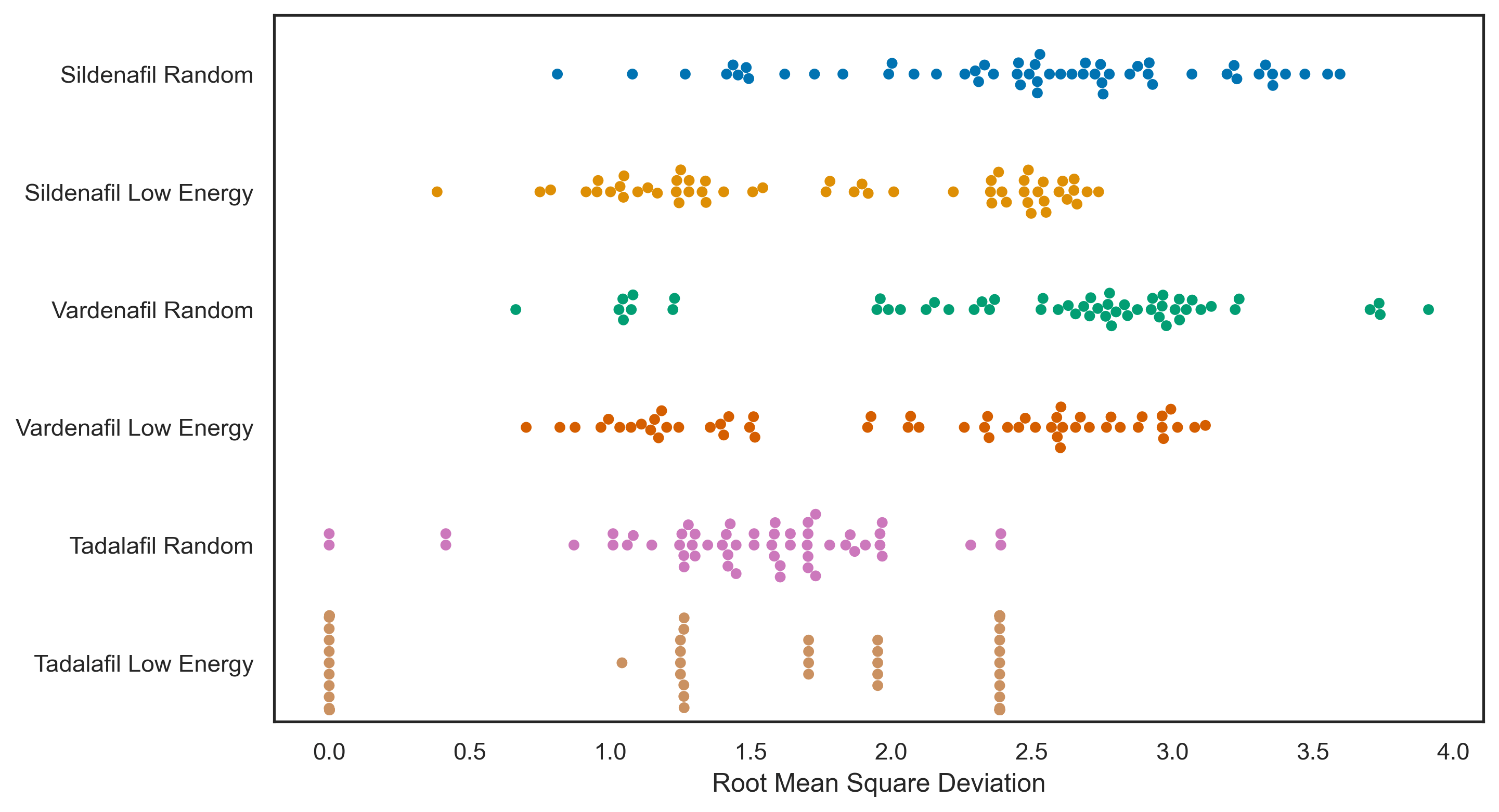}} %
	\hfill
	\subcaptionbox{Shape Similarity ($k=1$)}{\includegraphics[width=7.5cm]{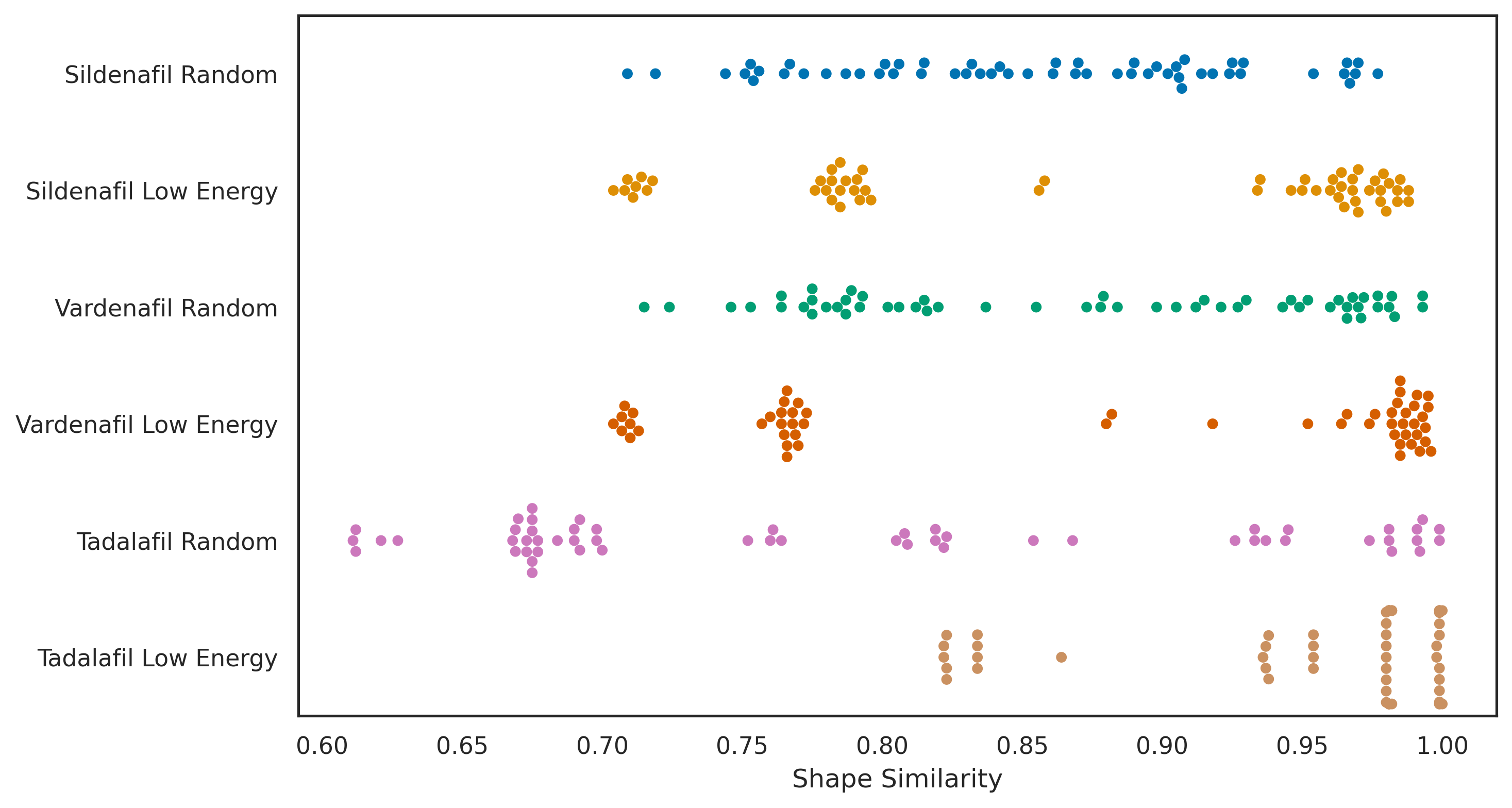}}%
	\\
	\subcaptionbox{Shape Similarity ($k=2$)}{\includegraphics[width=7.5cm]{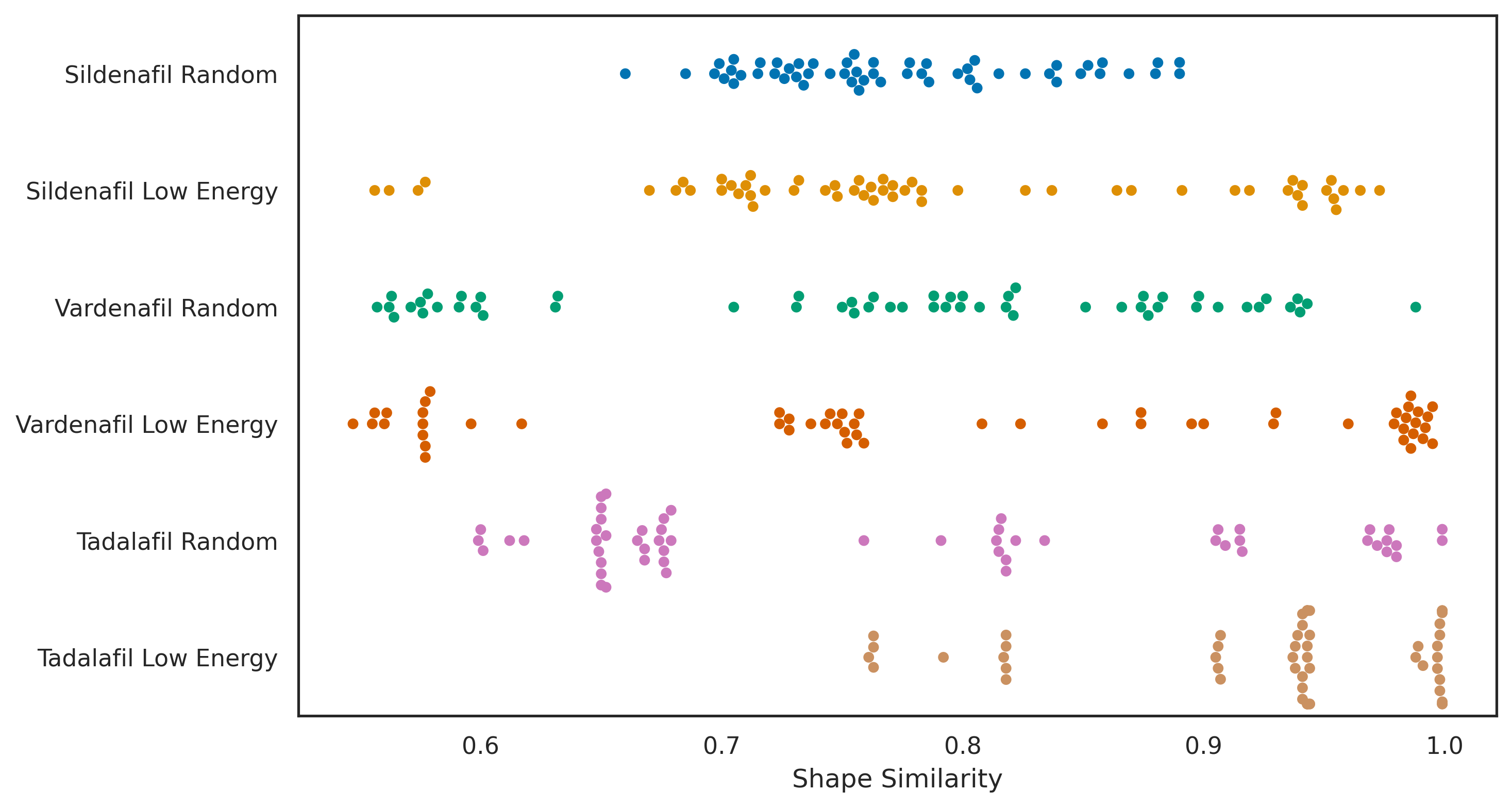}}%
	\caption{Swarm plots of the RMSD (in \AA) and shape similarity for our set of conformers highlight the general trend that different conformers are classed as having similar shape, despite significant variance in their atomic positions. Conformers with RMSD less than 1~\AA{} are considered similar, while those over 3~\AA{} have significant differences.} \label{fig:simvsrms}%
\end{figure}

\subsection{Comparison to Existing Methods}
The PDE5 inhibitor series was also used to investigate how well KQMolSA compares to the previous work, and to other open source shape similarity methods. Table~\ref{tab:SVTcomparison} provides the shape-similarity scores observed between the PDE5 inhibitors for KQMolSA (for $k=1$ and $k=2$), RGMolSA \cite{cole2022riemannian}, USRCAT \cite{Schreyer_Blundell_2012, landrum}, Shape-It \cite{Taminau_Thijs_De_Winter_2008} and MolSG \cite{SCPG}. A 2D representation, in the form of the 1024-bit Morgan fingerprint using radius 3, is also included. Each descriptor uses a similarity score between 0 (different) and 1 (identical). 
\vspace*{-20pt}
\begin{center}
\begin{table}[h]
\caption{Comparison of the work presented here (KQMolSA) to the previous work (RGMolSA)  \cite{cole2022riemannian} and existing atomic-distance \cite{Schreyer_Blundell_2012}, atomic-centred \cite{Taminau_Thijs_De_Winter_2008} and molecular surface based \cite{SCPG} descriptors. In all cases the similarity scores given are bound by 0 (no similarity) and 1 (identical).}
\vspace{10pt}
\begin{tabular}{|c|c|c|c|c|c|c|c|} 
\hline 
& \shortstack{\bf KQMolSA \\ (k=1)} & \shortstack{\bf KQMolSA \\ (k=2)} & \shortstack{\bf RGMolSA \\ {   } } & {\bf USRCAT} & {\bf Shape-It} & {\bf MolSG} & {\bf \shortstack{Morgan\\ Fingerprint}}\\
\hline
\shortstack{Sildenafil-\\Vardenafil} & 0.907 & 0.652 & 0.903 & 0.384 & 0.388 &  0.704 & 0.667 \\
\hline
\shortstack{Sildenafil-\\Tadalafil} & 0.449 & 0.482 & 0.809 & 0.269 & 0.278 &  0.746 & 0.201 \\
\hline
\shortstack{Vardenafil-\\Tadalafil} & 0.432 & 0.470 & 0.725 & 0.291 & 0.353 & 0.887  & 0.209 \\
\hline
\end{tabular}
\label{tab:SVTcomparison}
\end{table}
\end{center}
As discussed in the prequel to this paper, as Sildenafil and Vardenafil are close structural analogues they should display both high shape and fingerprint similarity. As Tadalafil is known to occupy a similar volume in PDE5 compared to the other inhibitors, we'd expect high shape similarity scores also, but lower 2D similarity. One conformer of each molecule is considered for simplicity. \\
\\
As for RGMolSA, Sildenafil and Vardenafil are scored as highly similar, with a score of 0.907 ($k=1$). However Tadalafil is not scored as highly, and for KQMolSA would be classed as dissimilar if the typical threshold of 0.7 was used. Lower similarity is observed for $k=2$, which is expected as discussed previously. The similarity score for $k=2$ has a small dependence on the order of comparison (A compared to B yields a score which may differ at the second decimal place from B compared to A, Table \ref{tab:k2_instability}). This is due to the distance calculation involving a numerical minimisation procedure rather than an exact expression, but this will have no practical implications in chemoinformatics applications. Both proposed methods (RGMolSA and KQMolSA) perform well in this simple study, with a higher predicted similarity for Sildenafil and Vardenafil than all the other 3D methods, and a more intuitive ordering of the relative similarity measures than MolSG. However, a full scale benchmarking study will be required to verify their performance.

\begin{center}
\begin{table}
\centering
\caption{Similarity scores for the PDE5 inhibitors for k=2 highlighting the dependence on the order of comparison.}
\begin{tabular}{|c|c|c|c|} 
\hline 
& \bf Sildenafil & \bf Vardenafil & \bf Tadalafil \\
\hline
Sildenafil & - & 0.652 & 0.462 \\
\hline
Vardenafil & 0.648 & - & 0.470 \\
\hline
Tadalafil & 0.482 & 0.470 & - \\
\hline
\end{tabular}
\label{tab:k2_instability}
\end{table}
\end{center}
\vspace{-20pt}
\subsection{Similarity to Potential Decoys}

As for RGMolSA, we also wanted to check how the method handles molecules that should be classed as genuinely different from the PDE5 inhibitor molecules. We therefore present a comparison to four other molecules (Figure~\ref{fig:others}): Arginine (supplement) which has a lower molecular weight, but similar general shape (a long chain of spheres); Lymecycline (antibiotic), with a higher molecular weight and a four-ring motif potentially giving part of the molecule a similar shape to Sildenafil; Diflorasone (topical corticosteroid), which has a similar molecular weight and four rings, but has a different therapeutic target/indication and S-octylglutathione (oligopeptide), which again has similar molecular weight, but no rings and the potential for similarity due to the branching in the centre of the molecule.

\begin{figure}[!ht]
\centering
\subcaptionbox*{Arginine}{\includegraphics[width=0.5\textwidth]{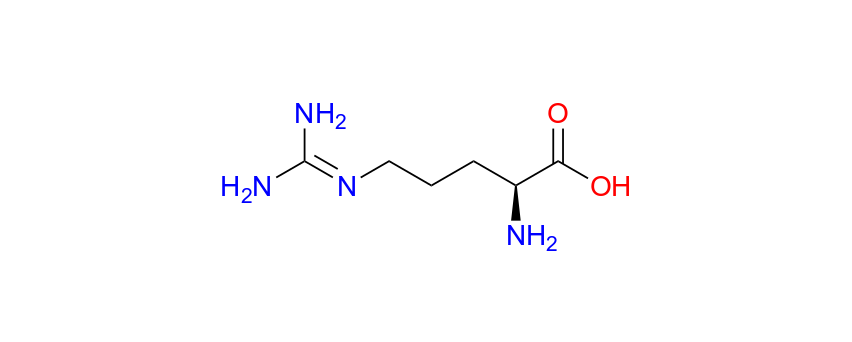}}%
\hfill
\subcaptionbox*{Lymecycline}{\includegraphics[width=0.5\textwidth]{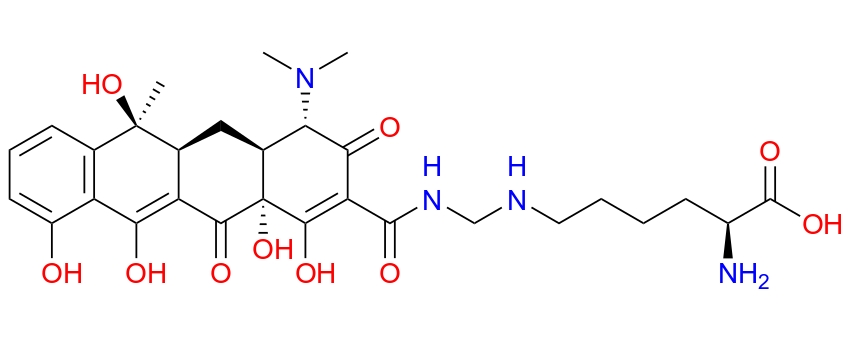}}%
\\
\subcaptionbox*{Diflorasone}{\includegraphics[width=0.5\textwidth]{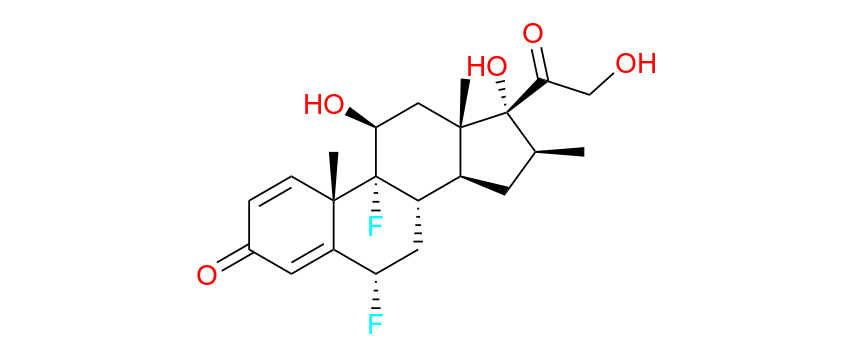}}%
\hfill
\subcaptionbox*{S-Octylglutathione}{\includegraphics[width=0.5\textwidth]{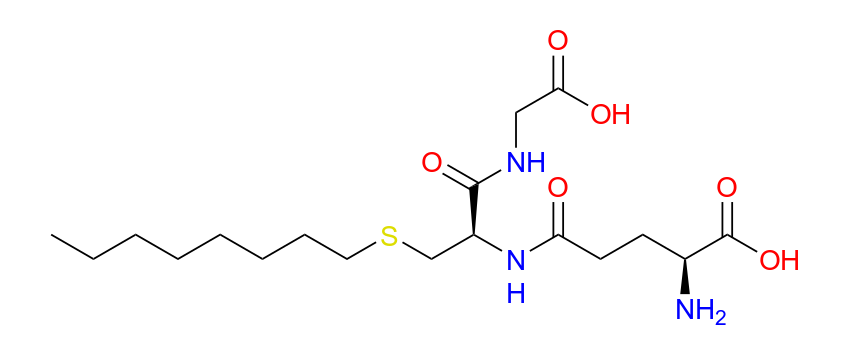}}%
\caption{Chemical structures of potential decoy molecules.} \label{fig:others}
\end{figure}

The results of this comparison are presented in Figure~\ref{fig:othersoverlay}. Most of the scores obtained for both $k=1$ and $k=2$ fall significantly below the typical threshold of 0.7 for similarity, and as such these molecules would be classed as genuinely different and likely inactive against PDE5. The exception is the comparison between Tadalafil and Diflorasone, where a higher score of 0.74 ($k=1$) is obtained. Due to the similarity between their structures (both contain a motif of 4 fused rings), we would expect to see some similarity between the two. Inspection by eye of both the space filling model and surface of the two molecules also suggests they do have genuinely similar shapes (Figure \ref{fig:tfil_dif}). These were also classed as potentially similar by RGMolSA (similarity of 0.872). 

\begin{figure}
\centering
\includegraphics[width=6.5in]{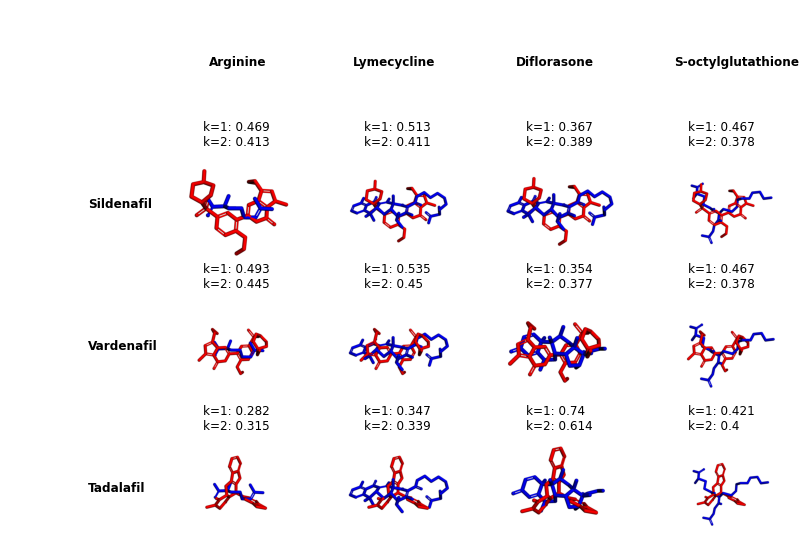}
\caption{KQMolSA similarity (for $k=1$ and $k=2$) of four `different' molecules (blue) to the PDE5 inhibitor test series (red). The overlay of the structures was computed using Open3DAlign \cite{tosco_balle_shiri_2011}}\label{fig:othersoverlay}
\end{figure}

\begin{figure}[!ht]
\centering
\subcaptionbox{Tadalafil - Space Filling Model}{\includegraphics[width=0.50\textwidth]{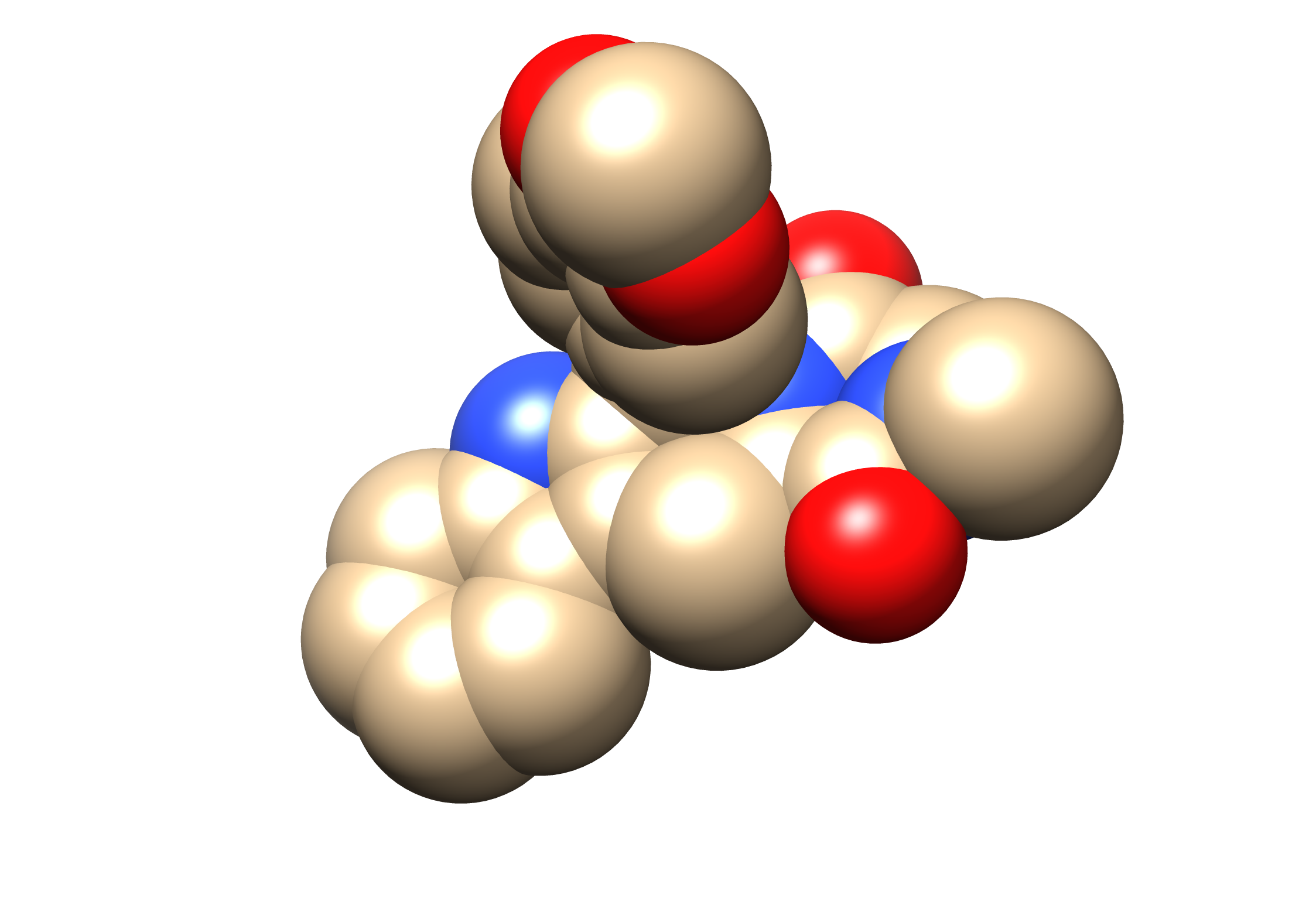}}%
\hfill
\subcaptionbox{Diflorasone - Space Filling Model}{\includegraphics[width=0.50\textwidth]{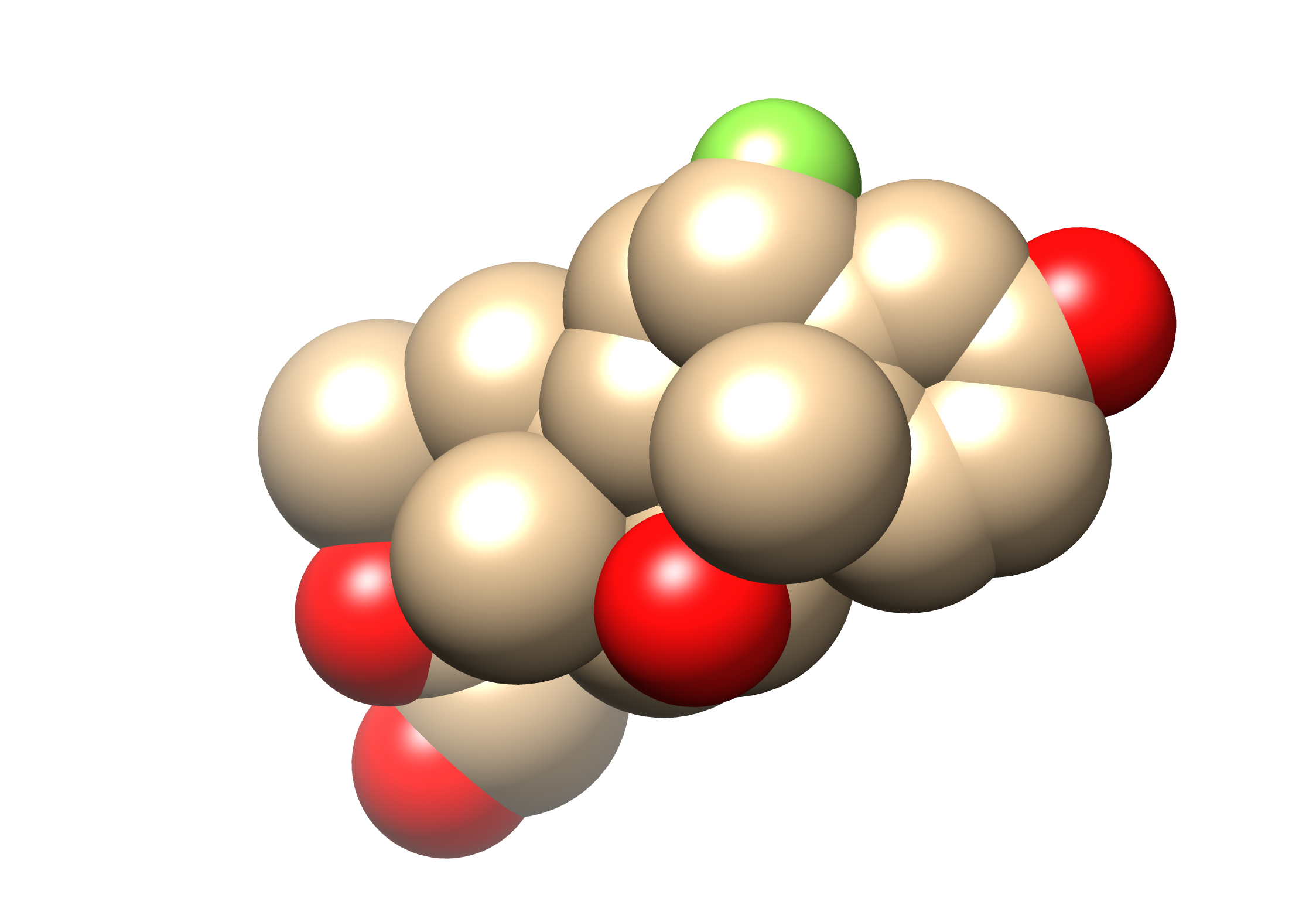}}%
\\
\subcaptionbox{Tadalafil - Surface}{\includegraphics[width=0.50\textwidth]{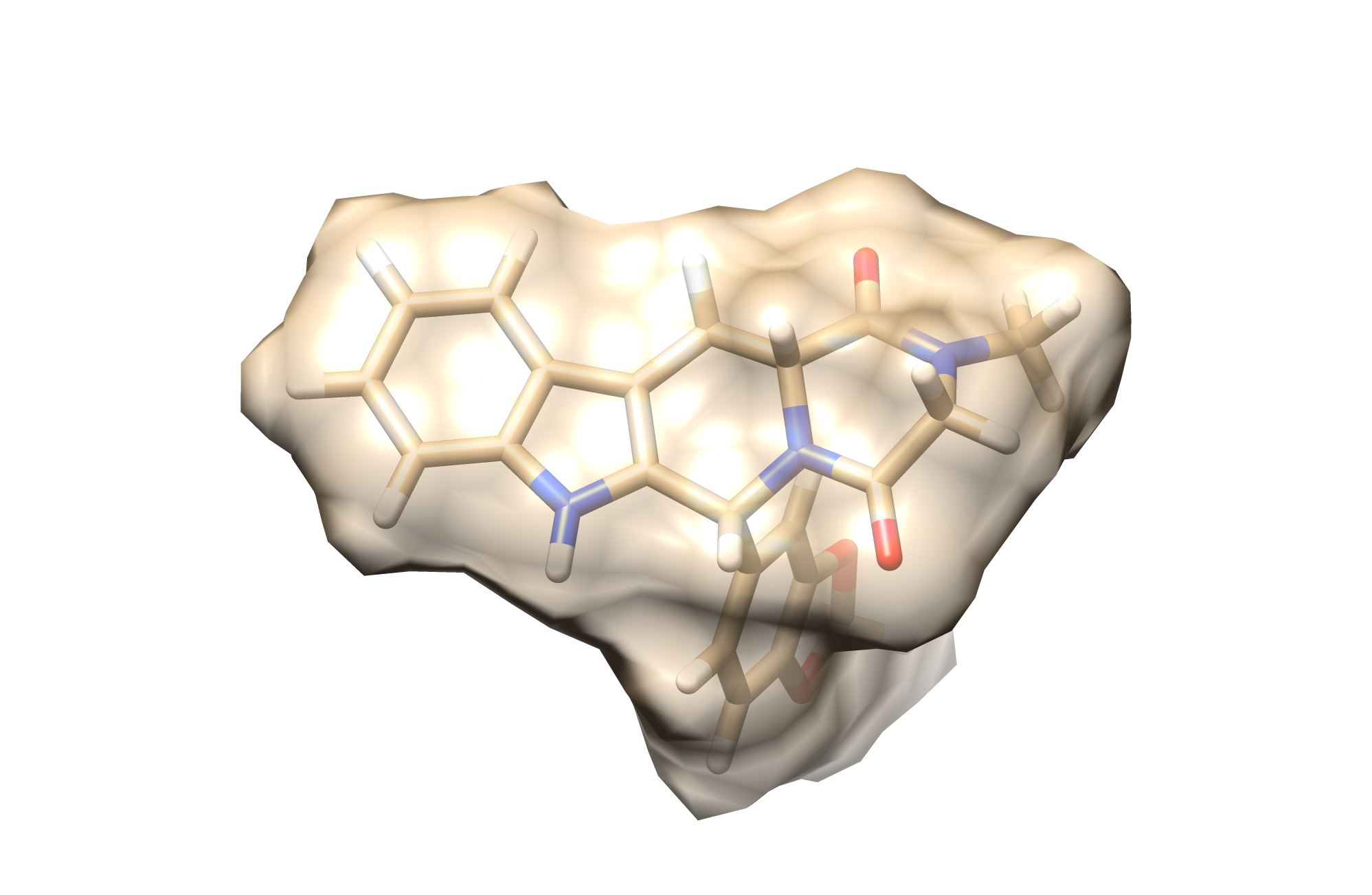}}%
\hfill
\subcaptionbox{Diflorasone - Surface}{\includegraphics[width=0.50\textwidth]{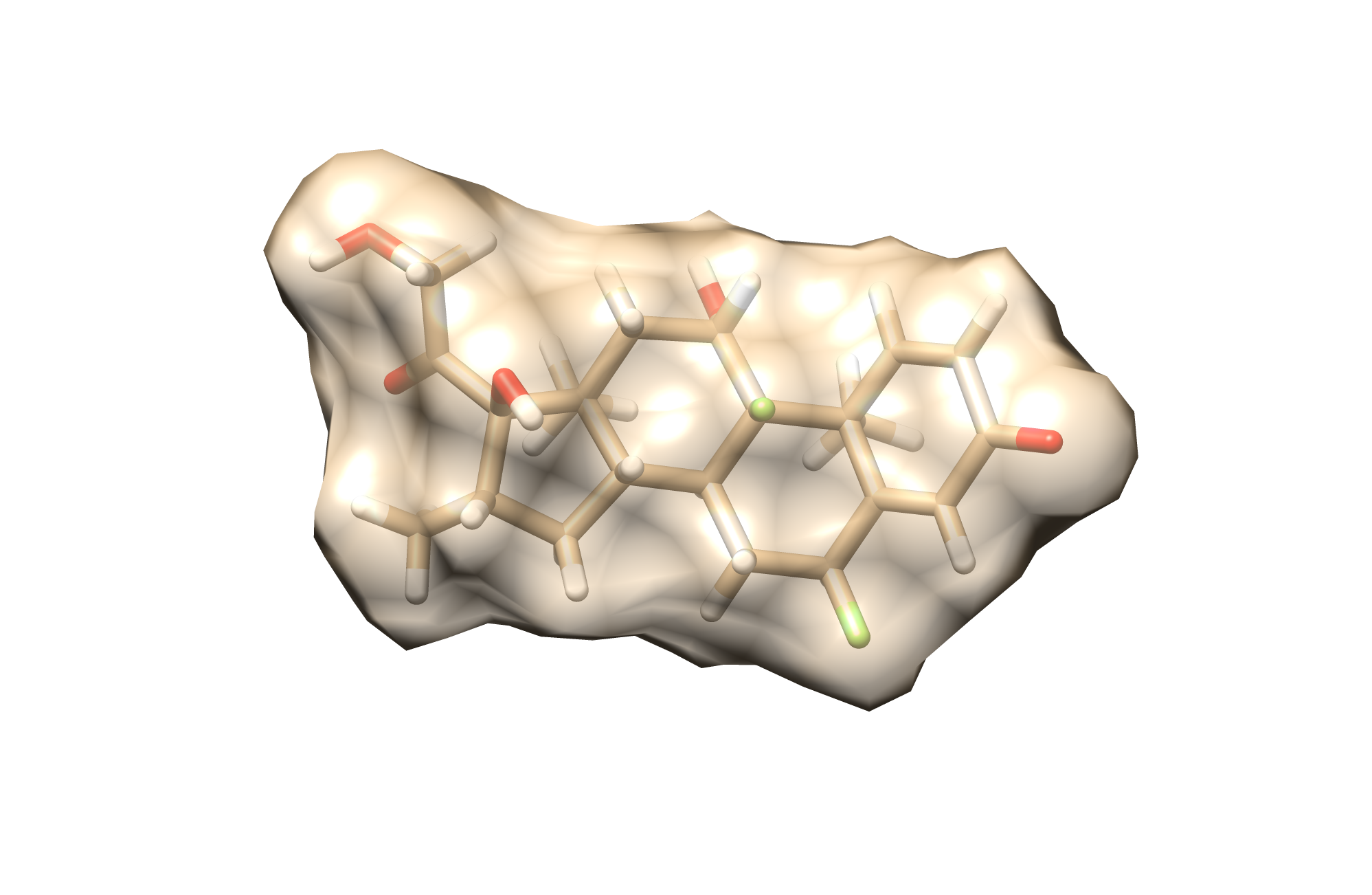}}%
\caption{Comparison by eye of both the space filling model and the surface of Tadalafil and Diflorasone highlights their similarity.} \label{fig:tfil_dif}
\end{figure}

\section{Conclusion}\label{section:conclusion}

We have outlined the theory underpinning an entirely novel shape descriptor,

\begin{equation}
\mathbb{M}_{ij} =  \iint_{\mathbb{C}}z^{i}\overline{z}^{j} e^{-k\varphi}F(z)\sqrt{-1}dz \wedge d\overline{z},   
\end{equation}

the $(2k + 1) \times (2k + 1)$ Hermitian matrix which captures the geometry of the molecular surface. The distance between two such matrix representations is then given as 

\begin{equation}
d(\mathbb{M}_{1},\mathbb{M}_{2})=k^{-\frac{3}{2}}\sqrt{\sum_{i=1}^{2k+1}(\lambda_{i}-\mu_{i})^{2}}.    
\end{equation}

An overall similarity score of 1 for identical molecules and 0 for no similarity is then obtained as 

\begin{equation} \label{eqn:simscore2}
   score(\mathcal{S}_{1},\mathcal{S}_{2}) = 0.3 (A_{min} / A_{max}) + 0.7 \frac{1}{1 + d(\mathbb{M}_{1},\mathbb{M}_{2})}.
\end{equation}

As with the previously reported work, the capabilities of KQMolSA were investigated using a series of PDE5 inhibitors known to have similar shape. The method generally handles conformers well, with similarity scores generally higher than $0.7$. The scores obtained were higher for $k=1$ than $k=2$, which is expected due to the greater detail leading to more sensitivity to changes in geometry. The insensitivity to deformation of the surface lead to RGMolSA outperforming KQMolSA in this area. KQMolSA performs relatively well compared to existing methods, identifying Sildenafil and Vardenafil as highly similar, but assigning lower similarity scores to Tadalafil. This small study suggests that RGMolSA might still perform better, but a full retrospective benchmarking study is required to confirm this. Compared to RGMolSA, KQMolSA does have the advantage of a lower dependence on the choice of base sphere. There may therefore be some instances where the use of KQMolSA is more appropriate despite its seemingly poorer performance, for example in the consideration of long chain molecules with few rings, where numerical errors are often observed for RGMolSA. Comparison to a set of potential decoy molecules yielded low scores for all except comparison of Tadalafil to Diflorasone, which were also classed as similar by RGMolSA. Inspection by eye of both the space filling and surface models of the molecules suggests that this assignment is reasonable, as they look similar in shape. Identification of such similarity evidences the potential for scaffold hopping by these methods.

Whilst the above tests suggest that the matrix $\mathbb{M}$ does give a promising description of molecular shape, the method does have some drawbacks, primarily in the calculation of the distance between two descriptors. While the notion of the distance between two Hermitian inner products (represented by the matrices $\mathbb{M}_{1}$ and $\mathbb{M}_{2}$) is well understood, the calculation of the distance between molecular surfaces requires the distance between a point on an $SL(2,\mathbb{C})$-orbit to be minimised. Despite the use of existing optimised minimisation algorithms, this process is still quite slow, depending on the extent of the required minimisation, and further does not guarantee that the global minimum has been found. This step typically takes a few seconds per pair, compared to a near-instantaneous calculation for RGMolSA. Further refinement of this step would be required for use of the method in screening ultra-large chemical libraries as part of a drug discovery pipeline.\\ 
\\
Of course, there are many other ways of measuring the distance between two Hermitian matrices.  One might hope that some form of machine learning, trained on an appropriate data set, might discern other useful geometries on the space of descriptors.

The method also contains numerical instability above $k=2$ (and for $k=2$ in a few instances), producing Hermitian matrices that are not positive definite. As Hermitian matrices differing only by a scale factor can be considered equivalent, we have handled such cases by scaling one matrix by a factor of 10-1000 to bring the eigenvalues into the range of Python's numerical tolerance.

Along with addressing these issues, both of the methods proposed could be further improved through the consideration of pharmacorphoric features, such as aromatic rings, hydrogen bond donors and acceptors, alongside the shape. As these features are important for binding, this may lead to improved predictions compared to the consideration of shape alone. As for RGMolSA, there would also be scope to investigate the use the Hermitian matrix descriptors produced by KQMolSA as a feature descriptor in machine learning.

\section{Appendix: Finding the K\"ahler potential}\label{sec:appendix}

Before giving the proof of the form of the K\"ahler potential, we dispense with a small technical point.  From the point of view of describing the K\"ahler form $\omega$ via
\[
\sqrt{-1}\partial \bar{\partial}\varphi = \omega,
\]
the K\"ahler potential $\varphi$ is only locally defined and adding any function $H$
satisfying $\sqrt{-1}\partial \bar{\partial}H = 0$

will also define a K\"ahler potential for $\omega$.  In our setting where the underlying complex manifold is $\mathbb{CP}^{1}$ and we are using the standard coordinate $z$, we can add any harmonic function $H:\mathbb{C} \rightarrow \mathbb{R}$ to obtain a valid K\"ahler potential.\\
\\
However, in K\"ahler Quantization, the potential $\varphi$ actually describes a global object, the Hermitian metric $h$ on the line bundle $L$. This means that the functions
\[
h(z^{j},z^{j}) = e^{-k\varphi(z)}|z|^{2j},
\]
are defined over whole sphere $\mathbb{CP}^{1}$. In particular, they extend to functions over the point at infinity.  For example the round metric has K\"ahler potential ${\varphi = -2\log(|z|^{2}+1)}$ and so, if we add a harmonic function $H$ we require
\[
\frac{|z|^{4k}}{(1+|z|^{2})^{2k}}e^{-kH}
\]
to be bounded.  The Liouville Theorem then implies $H$ must be constant.

\begin{theorem}[Form of K\"ahler potential]
Let $\omega$ be a K\"ahler metric of the form given by Equation~(\ref{eqn:form_of_KF}). If we denote the region corresponding to the $i^{th}$ sphere as $R_{i}\subset\mathbb{C}$, then the K\"ahler potential potential $\varphi$, which satisfies $\sqrt{-1}\partial\overline{\partial}\varphi=\omega$, is of the form
\[
\varphi(z) = \frac{C_{i}}{B_{i}}\log(|z-A_{i}|^{2}+B_{i})+\sum_{j=1}^{N}\mathcal{K}_{ij}\log(|\alpha_{ij}z+\beta_{ij}|^{2}),
\]
where $\mathcal{K} \in M^{N\times N}(\mathbb{R})$, and $\alpha, \beta \in M^{N\times N}(\mathbb{C})$.

\end{theorem}
\begin{proof}
The proof is by induction on the number of spheres $N$. For $N=1$ the metric $\omega$ is the round metric and we can take $\mathcal{K}=0$. Adding a new sphere to the surface changes the metric by adding a new region $R_{k}$ which is a disc where the metric takes the form 
\[
\omega(z)|_{R_{k}} = \frac{C_{k}}{(|z-A_{k}|^{2}+B_{k})^{2}}\sqrt{-1}dz \wedge d\overline{z}.
\]

We can map $R_{k}$ to the unit disc about the origin by a M\"obius transformation $\mathcal{M}$ in  such a way that, in the coordinate of the unit disc, the metric is given by

\[
\widetilde{\omega}(w) = \left\{ \begin{array}{ccc} F(w) \sqrt{-1}dw \wedge d\overline{w} & \mathrm{if} & |w| >1, \\
& & \\
\dfrac{\kappa}{(|w|^{2}+\varepsilon)^{2}} \sqrt{-1}dw \wedge d\overline{w} & \mathrm{if} & |w| \leq 1, 
\end{array} \right.
\]
for some function $F:\mathbb{C} \rightarrow \mathbb{R}$ and constants $\kappa, \varepsilon \in \mathbb{R}$.\\
\\
We solve the $\bar{\partial}$-equation using the Dolbeault method; for a compactly supported\footnote{Our function is not compactly supported but we could cut off at an arbitrary radius to produce such a function.} continuous function $H:\mathbb{C}\rightarrow\mathbb{C}$, 
\[
\psi(w) = \dfrac{1}{2\pi\sqrt{-1}}\iint_{\mathbb{C}} \dfrac{H(p)}{p-w}dp\wedge d\overline{p},
\]
solves $\overline{\partial}\psi = H(w) d\overline{w}$. We split the integral according to the form of the metric and consider
\[
\psi(w) = \dfrac{1}{2\pi\sqrt{-1}}\iint_{\mathbb{D}}\dfrac{\kappa}{(|p|^{2}+\varepsilon)^{2}(p-w)}dp\wedge d\overline{p} +\dfrac{1}{2\pi\sqrt{-1}}\iint_{\mathbb{C}\backslash \mathbb{D}}\dfrac{F(p)}{p-w}dp\wedge d\overline{p}.
\]
To compute the first integral we use the Cauchy--Pompeiu integral formula  
and the fact that
\[
\dfrac{\kappa}{(|p|^{2}+\varepsilon)^{2}} = \dfrac{\partial}{\partial \overline{p}}\left( \dfrac{(\kappa/\varepsilon)\overline{p}}{(|p|^2+\varepsilon)}  \right),
\]
to give
\[
\dfrac{1}{2\pi\sqrt{-1}}\iint_{\mathbb{D}}\dfrac{\kappa}{(|p|^{2}+\varepsilon)^{2}(p-w)}dp\wedge d\overline{p} = 
\]
\[
\left\{\begin{array}{ccc}
\left( \dfrac{(\kappa/\varepsilon)\overline{w}}{(|w|^2+\varepsilon)}  \right)-\dfrac{1}{2\pi\sqrt{-1}}\displaystyle\int_{\partial \mathbb{D}}\dfrac{(\kappa/\varepsilon)\overline{p}}{(|p|^2+\varepsilon)(p-w)}dp & \mathrm{if} & |w| < 1,\\
& &\\
-\dfrac{1}{2\pi\sqrt{-1}}\displaystyle\int_{\partial \mathbb{D}}\dfrac{(\kappa/\varepsilon)\overline{p}}{(|p|^2+\varepsilon)(p-w)}dp & \mathrm{if} & |w| > 1.
\end{array} \right.
\]
The contour integral 
\[
\dfrac{1}{2\pi\sqrt{-1}}\int_{\partial \mathbb{D}}\frac{(\kappa/\varepsilon)\overline{p}}{(|p|^2+B)(p-w)}dp,
\]
can be easily computed using the Cauchy Residue Formula and this yields
\[
\dfrac{1}{2\pi\sqrt{-1}}\int_{\partial \mathbb{D}}\frac{(\kappa/\varepsilon)\overline{p}}{(|p|^2+\varepsilon)(p-w)}dp = \left\{ \begin{array}{cc} 0 & \mathrm{if} \ |w|<1,\\
-\frac{(\kappa/\varepsilon)}{(1+\varepsilon)w} & \mathrm{if} \ |w|>1.
\end{array} \right.
\]
Finally, we arrive at
\[
\dfrac{1}{2\pi\sqrt{-1}}\iint_{\mathbb{D}}\dfrac{\kappa}{(|p|^{2}+\varepsilon)^{2}(p-w)}dp\wedge d\overline{p}= \left\{ \begin{array}{cc} \left( \frac{(\kappa/\varepsilon)\overline{w}}{|w|^2+\varepsilon}  \right) & \mathrm{if} \ |w|<1,\\
\frac{(\kappa/\varepsilon)}{(1+\varepsilon)w} & \mathrm{if} \ |w|>1.
\end{array} \right.
\]
To compute the second integral, we again split the domain and consider
\[
\dfrac{1}{2\pi\sqrt{-1}}\iint_{\mathbb{C}\backslash \mathbb{D}}\dfrac{F(p)}{p-w}dp\wedge d\overline{p} = \dfrac{1}{2\pi\sqrt{-1}}\iint_{\mathbb{C}}\dfrac{F(p)}{p-w}dp\wedge d\overline{p} - \dfrac{1}{2\pi\sqrt{-1}}\iint_{ \mathbb{D}}\dfrac{F(p)}{p-w}dp\wedge d\overline{p}. 
\]
The integral 
\[
S(w) = \dfrac{1}{2\pi\sqrt{-1}}\iint_{\mathbb{C}}\dfrac{F(p)}{p-w}dp\wedge d\overline{p},
\]
is a solution to
\[
\dfrac{\partial S}{\partial \overline{w}} = F(w).
\]
In the unit disc $\mathbb{D}$, $F$ has the form
\[
F(w) = \dfrac{\tilde{\kappa}}{(|w|^{2}+\tilde{\varepsilon})^{2}},
\]
where $\tilde{\kappa}$ and $\tilde{\varepsilon}$ are positive constants. Hence 
\[
\psi(w) = \left\{\begin{array}{ccc}
S(w)+\left(\dfrac{(\kappa/\varepsilon)}{|w|^{2}+\varepsilon}-\dfrac{(\tilde{\kappa}/\tilde{\varepsilon})}{|w|^{2}+\tilde{\varepsilon}} \right)\overline{w} & \mathrm{if} & |w|<1,\\
& & \\
S(w)+\left(\dfrac{(\kappa/\varepsilon)}{1+\varepsilon}-\dfrac{(\tilde{\kappa}/\tilde{\varepsilon})}{1+\tilde{\varepsilon}}\right)w^{-1} & \mathrm{if} & |w|>1,
\end{array}\right.
\]
solves $dw \wedge \overline{\partial}\psi = \widetilde{\omega}(w)$. 

If $Q(w)$ is a K\"ahler potential for $F(w)\sqrt{-1}dw \wedge d\overline{w}$ then 
\[
\widetilde{\varphi}(w) = \left\{ \begin{array}{ccc}
Q(w)+(\kappa/\varepsilon)\log(|w|^{2}+\varepsilon)-(\tilde{\kappa}/\tilde{\varepsilon})\log(|w|^{2}+\tilde{\varepsilon})-K &\mathrm{if} & |w|<1,\\
Q(w)+\left(\dfrac{(\kappa/\varepsilon)}{1+\varepsilon}-\dfrac{(\tilde{\kappa}/\tilde{\varepsilon})}{1+\tilde{\varepsilon}}\right)\log(|w|^{2}) &\mathrm{if} & |w|>1,
\end{array}\right.
\]
where 
\[
K = (\kappa/\varepsilon)\log(1+\varepsilon)-(\tilde{\kappa}/\tilde{\varepsilon})\log(1+\tilde{\varepsilon}),
\]
is a K\"ahler potential for $\widetilde{\omega}$. Pulling back the function $\widetilde{\varphi}$ via the M\"obius transformation
\[
\mathcal{M}(z) = \dfrac{\alpha z+\beta}{\gamma z+\delta}
\]
we see
\[
\varphi_{k}(z) = \left\{\begin{array}{ccc}
Q\left(\dfrac{\alpha z+\beta}{\gamma z+\delta} \right)+(\kappa/\varepsilon)\log\left(\left|\dfrac{\alpha z+\beta}{\gamma z+\delta}\right|^{2}+\varepsilon\right)-K
& \mathrm{if} & z\in R_{k}\\
Q\left(\dfrac{\alpha z+\beta}{\gamma z+\delta} \right)+\left(\dfrac{(\kappa/\varepsilon)}{1+\varepsilon}-\dfrac{(\tilde{\kappa}/\tilde{\varepsilon})}{1+\tilde{\varepsilon}}\right)\log\left(\left|\dfrac{\alpha z+\beta}{\gamma z+\delta}\right|^{2}\right) & \mathrm{if} & z \not \in R_{k}
\end{array}\right.
\]
is a K\"ahler potential for the metric which is singular at at the point $z=-\delta/\gamma$. We can replace the $Q$-term by the appropriate function for the previous $\varphi$ and then add the appropriate multiple of $\log(|\gamma z+\delta|^{2})$ to produce a K\"ahler potential of the appropriate form.

\end{proof}
\section{Acknowledgements}
The authors acknowledge support from an EPSRC Doctoral Training Partnership studentship (grant EP/R51309X/1), the Alan Turing Institute Enrichment Scheme (R.P.), and a UKRI Future Leaders Fellowship (grant MR/T019654/1) (D.J.C.). S.J.H. would like to thank Dr R. L. Hall for his interest and for useful conversations about the project. T.M. would like to thank University of California, Irvine for their hospitality whilst some of the work on this paper was completed.

\bibliographystyle{unsrt} 
\bibliography{references}
\end{document}